\documentclass[12pt,reqno]{amsart}

\usepackage[left=1in, right=1in, top=1.1in,bottom=1.1in]{geometry}
\usepackage[dvipsnames]{xcolor}
\usepackage[mathscr]{eucal}
\usepackage{latexsym,bm,amsfonts,amssymb,pifont,  bm,  stmaryrd}
\usepackage{graphicx}
\usepackage{color}
\usepackage{hyperref}
\hypersetup{
     colorlinks   = true,
     citecolor    = blue,
     linkcolor    = blue
}

\usepackage{pdfsync}
\usepackage[font={scriptsize}]{caption}

\usepackage{enumerate}
\usepackage{pdfsync}
\newcommand{\ep}{\varepsilon}

\newtheorem{theorem}{Theorem}[section]
\newtheorem{Def}[theorem]{Definition}
\newtheorem{thm}[theorem]{Theorem}

\newtheorem{cor}[theorem]{Corollary}
\newtheorem{lemma}[theorem]{Lemma}

\newtheorem{hyp}[theorem]{Hypothesis}
\theoremstyle{remark}
\newtheorem{remark}[theorem]{Remark}
\newtheorem{notation}[theorem]{Notation}

\numberwithin{equation}{section}

\def\RR{\mathbb{R}}

\def\mE{\mathbb{E}}
\def\EE{\mathbb{E}}

\def\PP {\mathbb{P}}

\newcommand{\cf}{{\mathcal F}}
\newcommand{\cg}{{\mathcal G}}

\newcommand{\cn}{{\mathcal N}}

\newcommand{\cp}{{\mathcal P}}

\def\si{\sigma}

\def\al{{\alpha}}

\def\ga{{\gamma}}

 \def\Si{{\Sigma}}

 \newcommand{\ka}{\kappa}

\newcommand{\lp}{\left(}
\newcommand{\rp}{\right)}
\newcommand{\lc}{\left[}
\newcommand{\rc}{\right]}

\def \eref#1{\hbox{(\ref{#1})}}

 \def\wt{\tilde}

\def \eref#1{\hbox{(\ref{#1})}}

\begin{document}
\title[Anticipative nonlinear filtering]
{On the anticipative nonlinear filtering problem \\ and its stability}
\date{}

\author[G. Lin]{Guang Lin}\author[Y. Liu]{Yanghui Liu}\author[S. Tindel]{Samy Tindel}


 \address{Guang Lin: Department of Mathematics,
Purdue University,
150 N. University Street,
W. Lafayette, IN 47907,
USA.}
\email{guanglin@purdue.edu} 
\address{Yanghui Liu: Department of Mathematics,
Purdue University,
150 N. University Street,
W. Lafayette, IN 47907,
USA.}
\email{liu2048@purdue.edu}
 \address{Samy Tindel: Department of Mathematics,
Purdue University,
150 N. University Street,
W. Lafayette, IN 47907,
USA.}
\email{stindel@purdue.edu}

 \begin{abstract}
 In this paper, we consider an anticipative nonlinear filtering problem, in which the observation noise is correlated with the past of the signal. This new signal-observation model has its applications in both finance models with insider trading and in engineering. We derive a new equation for the filter in this context, analyzing  both the nonlinear and the linear cases. We also handle the case of a finite filter with Volterra type observation. The performance of our algorithm is presented through numerical experiments.
\end{abstract}

\maketitle

{
\hypersetup{linkcolor=black}
 \tableofcontents
}

\section{Introduction}

In its most classical setting, the filtering problem can be summarized as follows: let $(\Omega, \cf, \PP)$ be a complete probability space, let $(\cf_t, \,t \geq 0)$ be an increasing family of sub $\sigma$-fields of $\cf$, and suppose that all $\PP$-null sets belong to $\cf_0$. We assume the existence of an underlying signal process $(X_t)_{t\geq 0}$ with values in  $\RR^{m}$ which can not be observed directly, and we are given an observation process $(Z_t)_{t\geq 0}$ with values in $\RR^{n}$ which is related to $(X_t)_{t\geq 0}$ and disturbed by the noise process $(N_t)_{t\geq 0}$. A main task in the   filtering theory is to estimate the signal process $(X_t )_{t\geq 0}$ based on the $(Z_{t})_{t\geq 0}$.  To give an example close enough to the situation which will be handled in the current paper, a model for the dynamics of both $X$ and $Z$ can be given as:
\begin{eqnarray}\label{eq:filtering-intro}
X_t&=&  X_{0} + \int^t_0 a(X_s)ds + W_t,\\
 Z_t &=& \int^t_0 h(X_s) ds + N_t, \notag
\end{eqnarray}
where $X$ and $Z$ are respectively the signal and the observation process, and $W$ and $N$ are  standard Brownian motions (or martingales), both adapted to the filtration $\cf_t$. In \eqref{eq:filtering-intro}, we also assume that the noises  $W$ and $N$ are decorrelated for convenience.

We refer to    \cite[Chapter 2-3]{Handbook}, \cite[Chapter 8]{K} and \cite{LS}  for a detailed account on the classical model \eqref{eq:filtering-intro} in full generality. Although we cannot give a real overview of the vast literature on filtering in this short introduction,  let us mention that a first basic step in order to solve the nonlinear filtering problem \eref{eq:filtering-intro} is to obtain an equation for the unnormalized conditional density of the observation $X_{t}$ given $Z$. This equation turns out to be a linear stochastic PDE called  Zakai's equation (see \cite{Zakai}),     which is one of the most classical objects studied in stochastic analysis. Then the numerical methods either rely on a particle representation of Zakai's equation (see \cite{DM, KX})  or on a direct discretization of Zakai's equation \cite{FL, Loto}. The reader is sent to \cite{Handbook} for the abundant literature on this topic. We should also mention a different direction which aims at taking into account possible time correlations for both the signal and observation noises. This is mainly achieved by considering a fractional Brownian motion model for the noises $W$ and $N$ (see e.g \cite{Coutin,Kleptsyna}), for which a complete answer in the nonlinear case is still a challenging issue.

In the current contribution, we wish to go back to one of the fundamental assumptions in equation \eqref{eq:filtering-intro}, namely the fact that the initial condition $X_{0}$ and the observation noise $N$ are independent (which follows from the fact that $N$ is $\cf_t$-Brownian motion). While this assumption seems to be natural at first sight, one can argue that this hypothesis is violated in many interesting situations.  We will go back to possible applications later in the introduction, but let us mention for the moment a classical target tracking problem for which misspecification on the initial condition has an influence on the observation noise (e.g via temperature, wind conditions or other environmental variables). We will thus see how the classical filtering problem is modified when $X_{0}$ and $N$ are correlated, by considering $X_{0}$ as an anticipative random variable (the reader might think of $X_{0}$ as a Wiener integral of the form $\int_{0}^{\infty} \phi_{s} \, dN_{s}$ for a deterministic function $\phi\in L^{2}(\RR_{+})$ in order to have a concrete example in mind).
This anticipative filtering problem has been considered only recently; see~\cite{oksendal}, where the authors assume that the signal $X_{t}=X_{0}$ is constant in time and correlated to the observation noise, and where the observation is linearly dependent on $X$.

The aim of this paper is thus to consider the aforementioned anticipative filtering problem in the general setting given by \eqref{eq:filtering-intro}, where we just assume the following general correlation property between the initial condition $X_{0}$ and the noise $N$.
\begin{hyp}\label{hyp.1}
Let $X_0$,  $W$ and $N$ be given as in equation \eqref{eq:filtering-intro}. We assume that the family $(X_0,W_t,  N_t, 0 \leq t \leq T)$ is Gaussian, such that  the correlation
\begin{equation}\label{eq:rho-N}
\rho_N(t) :=\mE( N_tX_0^T) , \qquad 0 \leq t \leq T
\end{equation}
is a   function in $C^{2}([0, T], \RR^{n\times m})$.
\end{hyp}
Within this general framework, we will focus on the following issue concerning the filtering problem:

\noindent
\emph{(i)}
In the general nonlinear case given by equation \eqref{eq:filtering-intro}, we derive a Zakai type equation for the unnormalized filter, as well as a Kushner-FKK type stochastic equation for the normalized filter. As the reader will see, the anticipative nature of our problem will affect all the coefficients of those equations.

\noindent
\emph{(ii)}
Whenever a linear situation is considered in the system \eqref{eq:filtering-intro}, we get a modified version of the Kalman-Bucy filter. As in the general case alluded to in (i), the coefficients of the linear filter are nontrivially affected by the correlation $\rho_{N}$. However, the Kalman-Bucy filter derived in Section \ref{section.linear} is very convenient to implement numerically.

\noindent
\emph{(iii)}
If we further assume that the correlation $\rho_{N}$ defined by \eqref{eq:rho-N} is compactly supported, then as expected,  we will be able to give some stability results for the anticipative Kalman-Bucy filter. More specifically, we show that for large times the difference between the anticipative and non-anticipative Kalman-Bucy filters converges exponentially fast to 0.

\noindent
\emph{(iv)}
Still in the linear case, we handle the case of a weighted Volterra type observation~$Z$ and get the corresponding expression for the anticipative finite dimensional optimal filter.  This should be seen as an alternative point of view on \cite{Coutin}, where a nonlinear Voterra type signal-observation system had been considered. Our method yields a straightforward implementation for the computation of the conditional mean and variance of the signal $X$.

\noindent
\emph{(v)}  Our simulation section will be focused on a classical radar-tracking example, where the vehicle's initial location is correlated with the observation noise. 
 The simulations show a significant improvement on the estimation accuracy whenever the anticipative filter is used.

\smallskip

\noindent
The basic technique we will resort to in order to deal with the anticipative filtering framework given by \eqref{eq:filtering-intro} can be summarized as follows.

\noindent
\emph{(a)}
We first rely on a general result concerning enlargements of filtrations. Namely, we will show that one can modify $N$ by a simple enough drift so that it becomes a $\cf_{t}$-Brownian motion (let us insist again on the fact that the system filtration $\cf$ includes the information on $X_{0}$). This general additional drift is then reflected in the coefficients of the filtering equations.  Note that the general enlargement of filtration result we invoke is interesting in its own right.  It can be seen as a multidimensional generalization of \cite{HU} and is detailed in Section \ref{section filtration}. 

\noindent
\emph{(b)}
Once the enlargement of filtration result is obtained, another key ingredient in the construction of the filter is a convenient introduction of some auxiliary signal process. Therefore, at the price of increasing slightly the dimension of our system and changing its coefficients, we will be able to go back to a more standard filtering setting.

\begin{notation}
For any integrable process $G_t$, $t \geq 0$, we denote by $\cf^{G}_{t}$, $t\geq 0$ the filtration $\si(G_{s}, s\in [0, t])$.  For convenience, we also   write $\hat{G}_t:=\mE [G_t| \cf^Z_t]$ and $ \tilde{G}_t := G_t -  \hat{G}_t$, where $Z$ is the observation process.
   \end{notation}

\section{Enlargement of filtration}\label{section filtration}

 When we consider an anticipative model like \eqref{eq:filtering-intro} under Hypothesis \ref{eq:rho-N}, one of the main problems is the following: while $N$ can be seen as a standard Brownian motion in $\cf^{N}$, it is no longer a Brownian motion with respect to the system filtration $\cf$.  Having this problem in mind,
in this section we show that there is a simple transformation of   $N$  that makes it a Brownian motion in the enlarged filtration.
This result will be first handled in a general framework. As mentioned in the introduction, it should be considered as a multidimensional generalization of    \cite{HU}.
\begin{lemma}\label{lem.filtration}
 Let $(B_t, 0 \leq t \leq T)$ be a standard $n$-dimensional Brownian Motion, and let $X = (X_1, X_2, \cdots, X_m)^T$ be a centered (mean zero) Gaussian random vector in $\RR^{m}$ with covariance~$\Sigma \in \RR^{m\times m}$. Suppose  that   $ \mE [B_t X^T] = \rho (t)$,
 where $\rho \in C^{2}([0, T]; \RR^{n\times m})$.  In addition, we assume that the family $\{X , B_{t}; 0\leq t \leq T\}$ is jointly Gaussian.
We consider a function $g(t) \in C([0,T]; \RR^{n\times m})$
 defined on $[0, T]$ (see Lemma \ref{lem.g} for the existence of such a function)   such that
\begin{eqnarray}\label{eqn.g}
g'(t)\lp\Sigma   -\int^t_0 \rho'(u)^T \rho'(u)  du\rp  &= \rho'(t)  , ~~~g(0)=0.
\end{eqnarray}
Let also $\lambda = \{\lambda(t, s); 0\leq s\leq t\leq t\}$ be the two-parameter $\RR^{n\times n}$-valued function defined by
\begin{eqnarray}\label{eqn.lam}
\lambda (t,s)   = g(t) p(s)+q(s),
\end{eqnarray}
 where $p \in C([0, T]; \RR^{m\times n}) $ and $q \in C([0, T]; \RR^{n\times n}) $ are respectively given by $p(s) =   \rho'' (s)^T$ and $ q(s) = -  g(s) \rho''(s)^T-g'(s) \rho'(s)^T  $.
  Then the process $\tilde{B} $ defined by
\begin{equation}\label{eqn.tb}
\begin{split}
\tilde{B}_t     = B_t - \int^t_0 \lambda (t,u)B_u du - g(t)X
\end{split}
\end{equation}
 is a $n$-dimensional $(\cg_t )$-Brownian motion, where
$ (\cg_t) $ is the augmented filtration $  \sigma(X, B_s; 0 \leq s \leq t)  $, $0 \leq t \leq T$.
 \end{lemma}
\begin{proof} We are going to   show that $\tilde{B}_t$ is a $\cg$-martingale. Then, taking into account    L\'evy's characterization and the continuity of $\tilde{B}$, we can  conclude  immediately  that $\wt{B}$ is a standard $\cg$-Brownian motion.

On a Gaussian space, it is well known that decorrelation implies independence. Therefore, in order to show that $\tilde{B}$ is a $\cg$-martingale, it suffices to show that
 \begin{equation}\label{eqn.tb.var}
\begin{split}
 \mE \{(\tilde{B}_t-\tilde{B}_s) X^T\}=0 \quad \text{ and } \quad \mE \{(\tilde{B}_t-\tilde{B}_s) B_r^T\}=0,\quad 0 \leq r \leq s \leq t.
\end{split}
\end{equation}

  Note that by  \eref{eqn.tb} and the identities $\mE [XX^T] = \Sigma  $ and $ \mE [B_t X^T] = \rho (t)$, the first equation of \eref{eqn.tb.var} is equivalent to
   \begin{eqnarray}
\label{eqn.tb.1}
0=& \lp\rho(t) -\int^t_0 \lambda (t,u) \rho(u)du -g(t) \Sigma\rp- \lp\rho(s)-\int^s_0 \lambda (s,u) \rho(u)du  - g(s) \Sigma\rp  .
\end{eqnarray}
We will now focus on the proof of \eref{eqn.tb.1}. To this aim observe that, due to the fact that   $\rho(0)=0$ and $g(0)=0$ (see relation \eref{eqn.g}), equation \eref{eqn.tb.1} is equivalent to
\begin{eqnarray}\label{eqn.tb.3}
0=&  \rho(t) -\int^t_0 \lambda (t,u) \rho(u)du -g(t) \Sigma\, .
\end{eqnarray}
We are now reduced to the proof of \eref{eqn.tb.3}.

Denote by $\partial_{s}$ the partial derivative with respect to the parameter $s$. Thanks to the definition \eref{eqn.lam} of $\lambda(t, s)$, the reader can easily check that $\lambda$ satisfies the following relation:
\begin{eqnarray}\label{eqn.lam2}
\lambda (t,s) = \partial_{s} \lp g(t)\rho'(s)^{T}  - g(s) \rho'(s)^{T}\rp \,.
\end{eqnarray}
Substituting \eref{eqn.lam2} into the right   side of \eref{eqn.tb.3} and then by    integration by parts  (together with the fact that $\rho (0)=0$) we obtain
\begin{eqnarray}\label{eqn.tb.31}
&& \rho(t) -\int^t_0 \lambda (t,u) \rho(u)du -g(t) \Sigma   \,= \,\rho(t)  + \int^t_0 [g(t) \rho'(u)^T - g(u) \rho'(u)^T]\rho'(u)du     -g(t) \Sigma
\nonumber
\\
  &&\quad\quad\quad =
  \rho(t) + \int_{0}^{t} [g(t) - g(u)] \rho'(u)^{T} \rho'(u) du - g(t) \Si
  .
\end{eqnarray}
We now apply   integration by parts  again to $U(u) = g(t)-g(u)$ and $dV(u) = \rho'(u)^{T} \rho'(u) du $, which yields
\begin{eqnarray}\label{eqn.tb.32}
 \int_{0}^{t} [g(t) - g(u)] \rho'(u)^{T} \rho'(u) du
  & =&  \int^t_0g'(u) \int^u_0\rho'(s)^T \rho'(s)ds du  .
\end{eqnarray}
Substituting \eref{eqn.tb.32} into \eref{eqn.tb.31}
and writing $g(t)\Si = \int_{0}^{t} g'(u)\Si du $
 we obtain
\begin{eqnarray}\label{eqn.tb.33}
 \rho(t) -\int^t_0 \lambda (t,u) \rho(u)du -g(t) \Sigma
 =
  \rho(t) +   \int^t_0g'(u) \lp
 \int^u_0\rho'(s)^T \rho'(s)ds-\Sigma
 \rp du .
\end{eqnarray}
Recalling that $g$ satisfies \eref{eqn.g}, it is now readily checked that the right-hand side of \eref{eqn.tb.33} vanishes. Hence we have proved our claim  \eref{eqn.tb.3}, which in turn proves the first assertion of relation \eref{eqn.tb.var}.

Let us turn to the  second equation of \eref{eqn.tb.var}.
Similarly to   \eref{eqn.tb.1} and recalling that $\lambda$ is a $\RR^{n\times n}$-valued function, the second equation of \eref{eqn.tb.var} can be written as
 \begin{eqnarray}\label{eqn.tb.2}
0&=& -\int^t_0 \lambda (t,u)\lc \begin{array}{ccc} u \wedge r &&0\\
 &\ddots&\\ 0&& u \wedge r\\ \end{array} \rc du + \int^s_0 \lambda (s,u)\lc \begin{array}{ccc} u \wedge r &&0\\
 &\ddots&\\ 0&& u \wedge r\\ \end{array} \rc du
\nonumber
 \\
 &&- g(t)\rho(r)^T + g(s)\rho(r)^T
\nonumber
 \\
 &=& -r\int^t_r \lambda (t,u)du - \int^r_0 \lambda (t,u) udu +r\int^s_r \lambda (s,u)du + \int^r_0 \lambda (s, u)udu
 \nonumber
 \\
 &&- g(t)\rho(r)^T + g(s)\rho(r)^T.
\end{eqnarray}

In the following, we show that   \eref{eqn.tb.2} holds. Note that,
owing to the fact that $g(0)=0$,
 it is clear that \eref{eqn.tb.2} holds for $r=0$, so   equation \eref{eqn.tb.2} is equivalent to
 \begin{equation}\label{eqn.tb.4}
\begin{split}
0=& -\int^t_r \lambda (t,u)du + \int^s_r \lambda (s,u) du - g(t)\rho'(r)^T +g(s)\rho'(r)^T ,\\
\end{split}
\end{equation}
 which is obtained by differentiating both sides of  \eref{eqn.tb.2}  with respect to $r$.
 Furthermore,
 along the same lines as for \eref{eqn.tb.3} and invoking the fact that $\int_{r}^{s} \lambda(s, u) du =0$ when $s=r$, it is easy to see  that equation \eref{eqn.tb.4} is equivalent to
 \begin{eqnarray}\label{eqn.lam3}
0=& -\int^t_r \lambda (t,u)du     - g(t)\rho'(r)^T +g(r)\rho'(r)^T .
\end{eqnarray}
Now relation \eref{eqn.lam3} follows immediately from   \eref{eqn.lam2}.
Plugging this information in our previous considerations, we get that \eref{eqn.tb.4} and therefore \eref{eqn.tb.2} hold true. This proves that   the second equation of \eref{eqn.tb.var} is satisfied. The proof is complete.
 \end{proof}

  With Lemma \ref{lem.filtration} in hand,
  notice that one can derive an explicit formula for $g(t)$ from \eref{eqn.g} if the symmetric matrix  $\Sigma - \int^t_0\rho'(s)^T \rho'(s)ds$ is non-singular.  The following result provides an equivalent condition to the non-singularity of   $\Sigma - \int^t_0\rho'(s)^T \rho'(s)ds$.
\begin{lemma}\label{lem.g}
Let the assumptions be as in Lemma \ref{lem.filtration}. Then

(a) The following identity holds for all $t\geq 0$:
\begin{eqnarray*}
\Sigma - \int^t_0\rho'(s)^T \rho'(s)ds= \mE \lp\lp X-\mE (X| \cf^B_t)\rp \lp X-\mE (X| \cf^B_t)\rp^T\rp  .
\end{eqnarray*}

(b) The matrix $\Sigma - \int^t_0\rho'(s)^T \rho'(s)ds$ is non-singular for all $t\in [0, T]$ whenever    $X \notin \cf^B_T$.

(c) If $\rho'(t_{0})\neq 0$ for some $t_{0}>0$, then $\Sigma - \int^t_0\rho'(s)^T \rho'(s)ds$ is non-singular for all $t\in [0,t_{0}]$.
\end{lemma}
\begin{proof}
Note first that
the $\cf^B$-Gaussian martingale $\mE (X| \cf^B_t)$ can be represented as a Wiener integral $\int^t_0 f_s dB_s$\,, where $f_s  \in L_{2}([0,T], \RR^{m\times n})$. By the definition of $\rho(t)$ it is easy to show that $ f_t  = \rho'(t)^T$. Therefore, we have
\begin{eqnarray*}
\mE \lp\lp X-\mE (X| \cf^B_t)\rp \lp X-\mE (X| \cf^B_t)\rp^T\rp
&=&
 \Si- \mE \lp\mE (X| \cf^B_t) \mE (X| \cf^B_t)^T\rp
\\
&=&   \Si-  \int^t_0  \rho'(s)^T \rho'(s) ds,
\end{eqnarray*}
which finishes the proof of our assertion (a).

We turn to the proof of (b).  Note   that the matrix $\mE \{[X-\mE (X| \cf^B_t)] [X-\mathbf{E}(X| \cf^B_t)]^T\}  $ is singular if and only if there exists a constant vector $(k_1, k_2, \cdots, k_m) \neq 0$  such that $  \Sigma_{i=1}^m k_i X_i  \in \cf^B_t$. In other words, $\mE \{[X-\mE (X| \cf^B_t)] [X-\mathbf{E}(X| \cf^B_t)]^T\}  $ is nonsingular if and only if  $  \Sigma_{i=1}^m k_i X_i  \notin \cf^B_t$ for any  $(k_1, k_2, \cdots, k_m)$. But this holds whenever $X  \notin \cf^B_t$. This completes the proof of (b). 

In order to prove (c), suppose now that $\rho'(t_{0})\neq 0$ for some $t_{0}>0$. Then invoking the continuity of $\rho'$ we can find a positive number $\ep>0$ such that $\rho'(s)\neq 0$ for $s\in [t_{0}, t_{0}+\ep]$. On the other hand, we have shown at the beginning of the proof that 
$$\EE(X|\cf^{B}_{t_{0}+\ep})= \int_{0}^{t_{0}+\ep}\rho'(s)dB_{s}=  \int_{0}^{t_{0} }\rho'(s)dB_{s}+ \int_{t_{0}}^{t_{0}+\ep}\rho'(s)dB_{s}.
$$
Consider now a given $t\in [0, t_{0}]$ and
note that since $\rho'(s)\neq 0$ for  $s\in [t_{0}, t_{0}+\ep]$, we have $\EE(X|\cf^{B}_{t_{0}+\ep})\notin \cf^{B}_{t}$. Therefore it is readily checked that $X= \EE(X|\cf^{B}_{t_{0}+\ep}) + (X-\EE(X|\cf^{B}_{t_{0}+\ep}))\notin \cf^{B}_{t} $. We can now apply directly item (b) and we conclude that $\Sigma - \int^t_0\rho'(s)^T \rho'(s)ds  $ is non-singular for all $t\in [0, t_{0}]$.
  The proof is complete.
\end{proof}
We can now give an explicit version for the relation  \eref{eqn.g} of $g$ under non-degeneracy assumptions in terms of $\rho$.
This follows immediately from Lemma \ref{lem.g} (c).
\begin{cor}\label{cor.g}
Let the assumptions be as in Lemma \ref{lem.filtration}. Let $T_{0}= \sup\{t\geq 0: \rho'(t)\neq 0 \}\in [0, \infty]$.
For $t<T_{0}$ we consider a function $g$ defined by $g(0)=0$ and
\begin{eqnarray*}
g'(t)  &=& \rho'(t) \lp\Sigma   -\int^t_0 \rho'(u)^T \rho'(u)  du\rp^{-1}.
\end{eqnarray*}
We also set  $g'(t)=0$ for $t\geq T_{0}$. Then the function $g$ satisfies equation \eref{eqn.g}. In particular, it is always possible to find a function $g$ such that \eref{eqn.g} is satisfied.
\end{cor}

\section{Anticipative filtering equation: Nonlinear case}

 In this section, we go back to the anticipative filtering problem \eref{eq:filtering-intro}, which is recalled here for the reader's convenience:
\begin{eqnarray}\label{eqn.1}
 X_t&=& X_{0}+\int_{0}^{t} a(X_s)ds +  W_t,
 \nonumber
 \\
 Z_t &=& \int^t_0 h(X_s) ds + N_t,
\end{eqnarray}
where
the family $(X_{0} , W_{t}, N_{t}; 0\leq t\leq T)$ satisfies Hypothesis \ref{hyp.1}. In particular, we assume  that  $   \rho_N(t) :=\mE( N_tX_0^T) $, $0 \leq t \leq T$  is a   function in $C^{2}([0, T], \RR^{n\times m})$.

  In order to derive an equation for the optimal filter, let us first see how Lemma \ref{lem.filtration} allows us to reduce our computations to an adaptive signal-observation system with modified coefficients.

   \begin{lemma}\label{lem.nmodel}
  Let $(X, Z)$ be the solution of \eref{eqn.1} and assume that Hypothesis \ref{hyp.1} is satisfied.
  Let $p$, $q$, $\lambda$, and $g$ be functions defined as in Lemma \ref{lem.filtration} with $\rho$ replaced by $\rho_{N}$.
  Set
  \begin{eqnarray}\label{eqn.nt}
\tilde{N}_{t} = N_{t} - \int_{0}^{t} \lambda (t,u)N_{u}du - g(t)X_{0}.
\end{eqnarray}
 The following statements holds:

  (a)  $\tilde{N}$ is a $\RR^{n}$-valued $\cf$-Brownian motion.

(b)  Consider the $\RR^{m}$-valued process $(\bar{X}_{t})_{t\leq T}$ defined by:
\begin{eqnarray}\label{eqn.xb}
  \bar{X}_{t} &=& X_{0}+ \int_{0}^{t}  p(s) N_{s}ds .
\end{eqnarray}
In addition, denote
\begin{eqnarray}\label{eqn.r}
r(t)=\lambda (t,t) = -g'(t)\rho'(t)^{T},
\end{eqnarray}
 and observe that $r(t) \in \RR^{n\times n}$ for all $t\leq T$. We define a process $(U_{t})_{t\leq T}$ such that $U_{t} \in \RR^{m}\times \RR^{m}\times \RR^{n}$
and some coefficients $b$, $\si$, $c$
 as follows:
\begin{eqnarray}\label{eqn.coeff}
U =
\lc
\begin{array}{c}
 X
\\
 \bar{X}
\\
 N
\end{array}
\rc,
\quad\quad
b (U_{t}) =
\lc
\begin{array}{c}
a(X_{t})
\\
p(t)N_{t}
\\
g'(t)\bar{X}_{t} + r(t)N_{t}
\end{array}
\rc,
\quad\quad
\si =
\lc
\begin{array}{c}
 I_{m}
\\
 0
\\
 0
\end{array}
\rc ,
\quad\quad
c =
\lc
\begin{array}{c}
 0
\\
 0
\\
 I_{n}
\end{array}
\rc .
\end{eqnarray}
Eventually, define a $\RR^{m}$-valued coefficient $k$ by:
\begin{eqnarray}\label{eqn.k}
k(U_{t}) &=& h(X_t)  + g'(t) \bar{X}_{t} +  r(t)N_t\,.
\end{eqnarray}
Then $(X, Z)$ satisfies a signal-observation system expressed in terms of $(U, Z)$:
\begin{eqnarray}
dU_{t} &=& b(U_{t})dt + cd\tilde{N}_{t} +\si dW_{t}
\label{eqn.U}
\\
dZ_{t} &=& k(U_{t}) dt + d\tilde{N}_{t} .
\label{eqn.model}
\end{eqnarray}

(c) The augmented    system \eref{eqn.U}-\eref{eqn.model} is now governed by $(W, \tilde{N})$, which is a $\cf$-Brownian motion.
   \end{lemma}
   \begin{proof}
   Item (a) follows from a direct application of Lemma \ref{lem.filtration} with $B=N$, $X=X_{0}$ and $\rho=\rho_{N}$.

   We turn to the proof of (b).  First,
   plugging \eref{eqn.nt} into \eref{eqn.1}
    we can write
  \begin{eqnarray}\label{eqn.z1}
 Z_t &=& \int^t_0 h(X_s) ds +\int^t_0 \lambda (t,s)N_s ds + g(t)X_{0} + \tilde{N}_t \,.
\end{eqnarray}
 Recall that
 in \eref{eqn.lam} we have defined $\lambda$ as
  $\lambda (t,s)   = g(t) p(s)+q(s)$. Therefore,  we can write the second and third terms of the right side of \eref{eqn.z1} as
\begin{eqnarray*}
\int^t_0 \lambda (t,s)N_s ds  +g(t)X_{0}
&=&
g(t) \lp\int_{0}^{t}  p(s) N_{s}ds +  X_{0}\rp +   \int^t_0 q(s)N_s ds
\\
&=& g(t) \bar{X}_{t} +\int_{0}^{t} q(s)N_{s}ds,
\end{eqnarray*}
where the second relation stems from \eref{eqn.xb}.
We now apply the elementary relation $\al_{t}\beta_{t}=\al_{0}\beta_{0} + \int_{0}^{t}\al_{s}'\beta_{s}ds + \int_{0}^{t} \al_{s} \beta_{s}' ds$ to the $C^{1}$-functions $\al_{t}=g(t)$ and $\beta_{t}=\bar{X}_{t}$, which  yields
\begin{eqnarray}\label{eqn.lamx}
\int^t_0 \lambda (t,s)N_s ds  +g(t)X_{0}&=&
\int_{0}^{t}g'(s) \bar{X}_{s}ds + \int_{0}^{t}g(s) p(s) N_{s}ds +   \int^t_0 q(s)N_s ds
\nonumber
\\
&=&
\int_{0}^{t}g'(s) \bar{X}_{s}ds + \int_{0}^{t}r(s)N_s ds,
\end{eqnarray}
where the last equality is due to the definition \eref{eqn.r} of $r$.
Reporting \eref{eqn.lamx} into \eref{eqn.z1}, and taking the definition \eref{eqn.k} of $k$ into account,
\begin{eqnarray*}
Z_t &=& \int^t_0 h(X_s) ds +
\int_{0}^{t}g'(s) \bar{X}_{s}ds + \int_{0}^{t}r(s)N_s ds
 + \tilde{N}_t
 \\
 &=&  \int_{0}^{t} k(U_{s}) ds +  \tilde{N}_{t} ,
\end{eqnarray*}
which is     equation in \eref{eqn.model}.

In the following, we derive the equation \eref{eqn.U} for $U$. Note again that by \eref{eqn.nt} and taking into account \eref{eqn.lamx} we obtain
\begin{eqnarray}\label{eqn.n}
       N_t &=&
 \int^t_0 \lambda (t,s)N_s ds  +g(t)X_{0} + \tilde{N}_t
 \nonumber
 \\
 &=&
  \int_{0}^{t}g'(s) \bar{X}_{s}ds + \int_{0}^{t}r(s) N_{s}ds     + \tilde{N}_t
   .
\end{eqnarray}
In order to get the equation for the process $U$ given by \eref{eqn.coeff}, it is thus sufficient to combine equation \eref{eqn.1} for $X$, equation \eref{eqn.xb} for $\bar{X}$ and relation \eref{eqn.n} for $N$.
\end{proof}

In Lemma \ref{lem.nmodel}, let us highlight again the fact that the new system \eref{eqn.U}-\eref{eqn.model} is governed by a $\cf$-Brownian motion $(W, \tilde{N})$. Hence we have reduced our anticipative problem to a classical filtering equation with modified coefficients. In order to give specific statements in this context, we now recall some basic notation.

\begin{notation}
In the context of Lemma \ref{lem.nmodel}, set
\begin{eqnarray*}
M_{t} =
\exp\lp \int_{0}^{t} k(U_{s})^{T}d Z_{s} -\frac12 \int_{0}^{t} |k(U_{s})|^{2} ds \rp.
\end{eqnarray*}
We also denote by $\tilde{\PP}$   the measure on $\Omega$ that is absolutely continuous with respect to $\PP$ with a  Radon-Nickodym derivative on $(\Omega, \cf_{t})$ given by:
 \begin{eqnarray*}
  \frac{d\tilde{\PP}}{d\PP}\Big|_{\cf_{t}} = M_{t}^{-1},
\end{eqnarray*}
where recall that $\cf_{t}  $ is the system filtration.
\end{notation}

With our modified setting in hand, the classical filtering results contained e.g. in \cite{X} yield the following result.
\begin{thm}
Let $(X, Z)$ be the solution of \eref{eqn.1} and let $U$ be as in Lemma \ref{lem.nmodel}. Then

(a) The  optimal filter $\pi_{t}:\langle \pi_{t}, f \rangle = \EE ( f( U_{t})| \cf_{t}^{Z} )$ satisfies the following nonlinear stochastic differential equation  on $[0, T]$:
\begin{eqnarray*}
\langle \pi_{t}, f \rangle =  \langle \pi_{0}, f \rangle +\int_{0}^{t} \langle \pi_{s}, L f \rangle   ds + \int_{0}^{t} (
\langle \pi_{s}, \nabla fc+f k^{T} \rangle -\langle \pi_{s}, f \rangle \langle \pi_{s}, k^{T} \rangle
 ) d\nu_{s},
\end{eqnarray*}
for all $f \in C^{2}_{b}(\RR^{d})$,
where $\nu$ is the innovation process defined by  $\nu_{t} = Z_{t} -\int_{0}^{t} \langle \pi_{s}, k \rangle  ds$ and where $\nabla f$ stands for the vector $\nabla f = (\partial_{1}f,\dots, \partial_{d}f )$.

 (b) 
Let $V_{t}$ be the unnormalized filter, defined by:
\begin{eqnarray*}
\langle V_{t}, f \rangle = \tilde{\EE} ( M_{t} f(U_{t})| \cf_{t}^{Z} ),
\end{eqnarray*}
where $\tilde{\EE}$ refers to the expectation with respect to the measure $\tilde{\PP}$. Then $V$   satisfies the following linear stochastic differential equation (usually called Zakai's equation):
\begin{eqnarray*}
\langle V_{t}, f \rangle = \langle V_{0}, f \rangle + \int_{0}^{t} \langle V_{s}, Lf \rangle ds + \int_{0}^{t} \langle V_{s}, \nabla f c + f k^{T} \rangle    dZ_{s} ,
\end{eqnarray*}
where the second order differential operator $L$ is defined by
\begin{eqnarray}\label{eqn.L}
Lf = \frac12 \sum_{i,j=1}^{d} A_{ij} \partial_{ij}^{2} f + \sum_{i=1}^{d} b_{i}\partial_{i} f.
\end{eqnarray}
In \eref{eqn.L}, we have also set
  $A= cc^{T}+\si\si^{T}$,   and $d$ is the dimension of $U$. 

\end{thm}


\section{Anticipative filtering equation: Linear case}\label{section.linear}

This section is devoted to a particularization of problem \eref{eqn.1} to a linear context. As usual in filtering theory, we will see that more explicit solutions to the filtering problem can be computed in this case. We also study the asymptotic stability of the filter in this framework.

\subsection{Filter equations}
Let us specify the filtering system we consider in this linear case. Namely, the couple $(X, Z)$ is assumed to satisfy the following system:
\begin{eqnarray}
 X_t&=& X_{0}+\int_{0}^{t} a(s)X_s ds + \si_{0} W_t,
 \label{eqn.linear1}
 \\
 Z_t &=& \int^t_0 h(s)X_s ds + N_t \,,
 \label{eqn.linear}
\end{eqnarray}
where $X_{t}\in \RR^{m}$, $Z_{t}\in \RR^{n}$, $W_{t}\in \RR^{l}$, $a(s)\in \RR^{m\times m}$, $\si_{0}\in \RR^{m\times l}$, and $h(s)\in \RR^{n\times m}$.

In our linear context, we still define
  $\bar{X}$   by \eref{eqn.xb}. We also define an
  $\RR^{2m+n}$-valued
   augmented signal $U$ and some augmented coefficients $b$, $\si$, $c$ and $k$ by
\begin{eqnarray*}
U =
\lc
\begin{array}{c}
 X
\\
 \bar{X}
\\
 N
\end{array}
\rc,
\quad\quad
b (t) =
\lc
\begin{array}{ccc}
a(t) &0&0
\\
0&0&p(t)
\\
0&g'(t)& r(t)
\end{array}
\rc,
\quad\quad
\si =
\lc
\begin{array}{c}
 \si_{0}
\\
 0
\\
 0
\end{array}
\rc ,
\quad\quad
c =
\lc
\begin{array}{c}
 0
\\
 0
\\
 I_{n}
\end{array}
\rc ,
\end{eqnarray*}
and
\begin{eqnarray*}
 k(t) = ( h(t) ,  g'(t) ,  r(t) ),
\end{eqnarray*}
where $a$ is defined by \eref{eqn.linear1}, $p$ and $g$ are introduced in Lemma \ref{lem.filtration} and $r$ is given in Lemma~\ref{lem.nmodel}.
It then follows from Lemma \ref{lem.nmodel} that the linear system \eref{eqn.linear} is equivalent to the following regular Kalman-Bucy signal-observation system:
\begin{eqnarray}\label{eqn.nlinear}
dU_{t} &=& b(t) U_{t} dt + cd\tilde{N}_{t} +\si dW_{t}
\nonumber
\\
dZ_{t} &=& k(t) U_{t}  dt + d\tilde{N}_{t} .
\end{eqnarray}
Let us  also
recall that
for the linear filtering problem \eref{eqn.nlinear}, the optimal filter $\pi_{t}$ is obtained as the following regular conditional law:
\begin{eqnarray}\label{eqn.pilinear}
\pi_{t} = \cn (\hat{U}_{t}, P_{t}),
\end{eqnarray}
where $\hat{U}_{t} = \EE[U_{t}| \cf_{t}^{Z}]$ and $P_{t} = \EE( (U_{t}-\hat{U}_{t}) (U_{t}-\hat{U}_{t})^{T} | \cf_{t}^{Z} )$ designates the conditional variance of $U_{t}$ given $\cf_{t}^{Z}$. The following theorem specifies the expressions of $\hat{U}$ and $P$:
\begin{thm}\label{thm.linear}
Let $\hat{U}$ and $P$ be the conditional mean and covariance of $U$ given by \eref{eqn.pilinear}. Then:

(i) $\hat{U}_{t}$ solves the equation
\begin{eqnarray}\label{eqn.uh}
\hat{U}_{t} &=& \hat{U}_{0} + \int_{0}^{t} b(t) \hat{U}_{s}ds + \int_{0}^{t} (c+P_{s}k(s)^{T})d\nu,
\end{eqnarray}
where the innovation process $\nu$ is given by $\nu_{t} = Z_{t} - \int_{0}^{t} k(s) \hat{U}_{s}ds $.

(ii)
The $\RR^{2m+n, 2m+n}$-valued conditional variance $P$ satisfies a Riccati equation of the form:
\begin{eqnarray}\label{eqn.P}
P_{t}' = P_{t} b(t)^{T} + b(t)P_{t} +A - (c+P_{t}k(t)^{T})(c+P_{t}k(t)^{T})^{T} ,
\end{eqnarray}
where $A=cc^{T}+\si\si^{T}$ as in equation \eref{eqn.L}.
\end{thm}
\begin{proof}
Once expression \eref{eqn.nlinear} is given for the augmented Kalman filter, our result is obtained as in the standard case, see e.g. \cite[Chapter 9]{X}.
\end{proof}

\subsection{Asymptotic stability}
We now particularize our situation to a linear context with constant coefficients. That is, we consider the following signal-observation system:
\begin{eqnarray}\label{eqn.const}
 X_t&=& X_{0}+\int_{0}^{t} a X_s ds +  W_t,
 \nonumber
 \\
 Z_t &=& \int^t_0 h X_s ds + N_t \,,
\end{eqnarray}
where $a\in \RR^{m\times m}$ and $h\in \RR^{n\times m}$ are constant matrices.
In order to state and prove our asymptotic stability result, we first need to recall some classical notions which can be found in \cite[Theorem 4.11]{KS} or \cite[Chapter 9]{X}.

\begin{Def}
In the following, $A$ stands for a $m\times m$ matrix, while $D\in\RR^{n\times m}$ and $B\in\RR^{m\times l}$ for given integers $l,m,n$.

\noindent
(i)
We define \emph{the stable subspace of matrix $A$}  as the direct sum of the (right) kernels of $(\lambda_{i}I-A)^{m_{i}}$, where $\lambda_{i}$ are negative eigenvalues of $A$ and $m_{i}$ is the multiplicity of $\lambda_{i}$.
We define \emph{the unstable subspace of $A$}  as the orthogonal of the stable   subspace of $A$.
 
\noindent
(ii)
 We call the couple of matrices $(A,D)$ \emph{detectable} if the  (right) kernel of 
 $$
 \lc\begin{array}{c} D\\DA\\ \vdots \\ DA^{m-1}\end{array}\rc   
 $$ 
is contained in the stable subspace of $A$. 

\noindent
(iii) 
We call   the couple of matrices $(A, B)$   \emph{stabilizable} if the unstable subspace of $A$ is contained in the  linear space spanned by the columns of $(B, AB, \cdots, A^{m-1}B)$.

\end{Def}

With the preliminary notions we have just introduced, we can state a result about existence of solutions for Riccati equations.
\begin{lemma}\label{lem.stable}
Let $a$ and $h$ be the coefficients given in \eref{eqn.const}. We assume that $(a, h)$ is detectable and $(a, I)$ is stabilizable. Then

(i) The algebraic Riccati equation 
\begin{equation}\label{eqn.ga}
\begin{split}
  \gamma_{\infty} a^T + a \gamma_{\infty} + I - \gamma_{\infty} h^Th\gamma_{\infty}
=0
 \end{split}
\end{equation}
 admits a unique solution $\ga_{\infty}$ in $\mathbb{R}^{m\times m}$.

(ii) We have
\begin{eqnarray}\label{eqn.lam0}
\lambda_{0}\equiv
\inf \{ -Re \lambda: \lambda \text{ is an eigenvalue of the matrix $a-\ga_{\infty}h^{T}h$} \} >0.
\end{eqnarray}

\end{lemma}

In the classical situation (i.e. for $X_{0}$ independent of $N$) and for a system like \eref{eqn.const}, it is well-known that the optimal filter   $\hat{X}_{t} = \EE(X_{t}|\cf^{Z}_{t})$ converges exponentially fast to
the solution of the algebraic Riccati equation \eref{eqn.ga}.
Our aim now is to prove that this convergence still holds true when the covariance between $X_{0}$ and $N$   vanishes in finite time.

 \begin{theorem}
 Consider the signal-observation system given by \eref{eqn.const} under the same conditions as for Lemma \ref{lem.stable}. We also assume that  there exists $T_{0}>0$ such that $\rho'(t)=0$ for all $t> T_{0}$. 
 Then

 (a)
 Let $P_{t}^{11}$ be the conditional variance of $X_{t}$ given $\cf_{t}^{Z}$. For all $\lambda<\lambda_{0}$, where $\lambda_{0}$ is defined by \eref{eqn.lam0}, we have
 \begin{eqnarray}\label{eqn.P11.lim}
\lim_{t\to\infty} e^{\lambda t} (P^{11}_{t} - \ga_{\infty}) =0.
\end{eqnarray}

(b)
Let $(\hat{X}^{0}_{t})$ be the optimal conditional expectation of $X_{t}$ given the observation obtained when $X_{0}$ and $N$ are independent (that is, $\rho_{N}=0$). Then for all   $\lambda<\lambda_{0}$
we have
  the convergence
 $ \lim_{t\to\infty} e^{\lambda t} (\hat{X}_{t} -  \hat{X}^{0}_{t}) = 0 $ almost surely.

(c) Let $\pi_{t}$ be the conditional Gaussian probability measure with mean $\hat{X}_{t}$ and covariance matrix  $P_{t}$.
In the same way, we define a conditional Gaussian measure $\pi^{0}_{t}$ given as $\pi^{0}_{t} = \cn (\hat{X}^{0}_{t}, P^{0}_{t})$,
where $\hat{X}^{0}_{t}$ is defined in item (b) above and $P^{0}_{t}$ is the conditional covariance of the usual Kalman filter(see e.g. \cite[Equation (9.8)]{X}).   Then we have
\begin{eqnarray*}
\lim d_{W}(\pi_{t} , {\pi}^{0}_{t}) \to 0, \text{ almost surely,}
\end{eqnarray*}
where $d_{W}$ denotes the Wasserstein metric in the space of probability measures.
 \end{theorem}

\begin{proof}
It follows from Theorem \ref{thm.linear} that the filter equations for
the augmented version of equation
 \eref{eqn.const} is:
\begin{eqnarray}
d\hat{U}_{t} &=& ( b(t)  - ck(t) -P_{t}k(t)^{T}k(t)  ) \hat{U}_{t}dt + (c+P_{t} k(t)^{T}) d {Z_{t}},
\label{eqn.U.const}
\\
\dot{P}_{t}  &=& P_{t} (b(t) - ck(t))^{T} + (b(t) - ck(t)) P_{t} +\si\si^{T} - P_{t} k(t)^{T}k(t)P_{t} ,
\label{eqn.P.const}
\end{eqnarray}
where $U$, $b$, $\si$ and $c$ are respectively defined by:
\begin{eqnarray*}
U =
\lc
\begin{array}{c}
 X
\\
 \bar{X}
\\
 N
\end{array}
\rc,
\quad\quad
b (t) =
\lc
\begin{array}{ccc}
a  &0&0
\\
0&0&p(t)
\\
0&g'(t)& r(t)
\end{array}
\rc,
\quad\quad
\si =
\lc
\begin{array}{c}
 I_{m}
\\
 0
\\
 0
\end{array}
\rc ,
\quad\quad
c =
\lc
\begin{array}{c}
 0
\\
 0
\\
 I_{n}
\end{array}
\rc ,
\end{eqnarray*}
and $k(t)$ is the matrix given by
\begin{eqnarray*}
 k(t) = [ h  \quad  g'(t) \quad  r(t) ].
\end{eqnarray*}
Observe that the following elementary identities hold true:
\begin{eqnarray}\label{eqn.ck}
ck(t) =  \lc\begin{array}{ccc} 0&0&0\\0&0&0\\h  &g'(t)&r(t) \end{array}\rc ,
\qquad
b(t)-ck(t) =  \lc\begin{array}{ccc} a &0&0\\0&0&p(t)\\-h  &0&0 \end{array}\rc.
\end{eqnarray}
Moreover, recalling that $P_{t}$ is a $\RR^{(2m+n)\times (2m+n)}$ matrix, we decompose $P_{t}$ into blocks of size $k\times l$ with $k, l\in \{m, n\}$ according to the $3$ components of $U$.
Hence, projecting equation~\eref{eqn.U.const} on the $X$ component and recalling \eref{eqn.ck}, it is readily checked that $\hat{X}_{t}$ satisfies
 \begin{eqnarray}\label{eqn.xhat.const}
d\hat{X}_{t} &=& a\hat{X}_{t} dt - (P^{11}h^{T} + P^{12}g'(t)^{T} + P^{13} r(t)^{T} ) (h \hat{X}_{t}  + g'(t) \hat{\bar{X}}_{t} + r(t) \hat{N}  ) dt
\nonumber
\\
&& +    (P^{11}h^{T} + P^{12}g'(t)^{T} + P^{13} r(t)^{T} )  dZ_{t}.
\end{eqnarray}
In the same way, projecting relation \eref{eqn.P.const} on the first component of $U$, we obtain that   $P^{11} = \EE ( (X_{t} - \hat{X}_{t})(X_{t} - \hat{X}_{t})^{T})$ is
\begin{eqnarray}\label{eqn.P11}
 \dot{P}^{11} &=& P^{11}a^{T}+ a P^{11} +I
\nonumber
\\
&& \quad-  (P^{11}h^{T}+ P^{12} g'(t)^{T}+P^{13}r(t)^{T}) (h P^{11}+g'(t)P^{21}+r(t)P^{31}).
\end{eqnarray}

Let us recall that the expression for $g'$ is obtained in Corollary \ref{cor.g} and $r$ is defined by \eref{eqn.r}. Therefore, since    $\rho'(t)=0$ for $t\geq T_{0}$, we easily get that for $t>T_{0}$ we also have $g'(t) =0$ and $r(t)=0$. Plugging this information into \eref{eqn.xhat.const} and \eref{eqn.P11}, the equations for $\hat{X}$ and $P^{11}$ becomes:
\begin{eqnarray*}
d \hat{X}_{t} &=&  a\hat{X}_{t}dt - P_{11} h^{T} h \hat{X}_{t} dt + P_{11}h^{T} dZ_{t}\,,
\\
\dot{P}^{11} &=&P^{11}a^{T}+ a P^{11} +I - P^{11}h^{T} h P^{11} .
\end{eqnarray*}
According to Lemma \ref{lem.stable}, if $(a, h)$ is detectable and $(a, I)$ is stabilizable, then equation  \eref{eqn.ga} has a unique solution. Furthermore, it is   shown in   \cite[Theorem 4.11]{KS}
that under the same conditions we have
\begin{eqnarray*}
\lim_{t\to \infty} e^{\lambda t} (P^{11}_{t}-\ga_{\infty}) =0,
\end{eqnarray*}
which is our claim \eref{eqn.P11.lim}.
Observe also that   Lemma \ref{lem.stable} implies that the matrix $a-\ga_{\infty}h^{T}h$ is asymptotically stable.
Items (b) and (c) in our Theorem thus follow    from the results in \cite[Section 2]{OP} (see also \cite[Section 9.5]{X}).
\end{proof}

\section{A finite filter}

In this section, we consider another application of the methods used for the anticipative filter \eref{eq:filtering-intro}. Namely, inspired by e.g. \cite{Coutin, Kleptsyna}, we wish to handle the case of a weighted Volterra type observation.

To be more specific, we are now considering a signal $(X_{t})_{t\leq T}$ and an observation $(Z_{t})_{t\leq T}$ governed   by the stochastic differential equations
\begin{eqnarray}
  X_t &=&  X_{0} + \int_{0}^{t} a(s)X_s   ds  +        W_t ,
  \label{eqn.vol.signal}
  \\
    Z_t &=&     \int^t_0 H (t, s)  X_s ds
     +        N_s,
     \label{eqn.vol}
\end{eqnarray}
where $(X_0, W_t,N_t)$ is a Gaussian family and the three terms are mutually independent.
As in the previous sections, we
 assume that $(W_{t}, N_t)$ is a standard Brownian motion, and $a: [0, T] \rightarrow \mathbb{R}^{m \times m},        H : [0, T]  ^2 \rightarrow \mathbb{R}$ are continuous functions.
   The observation information is given by the filtration of the observation process: $\cf^Z_t   = \sigma(Z_s; s \leq t)$, $t \in [0, T]$.
   The initial condition $X_{0}$ is assumed to be independent of $N$ in this section. However, the fact that $Z_{t}$ is governed by a Volterra type dynamics will force us to resort to the same augmented filtering equation as in the anticipative case. Observe that an anticipative initial condition in \eref{eqn.vol.signal} could also be treated with our methods. We have refrained from going in this direction for sake of conciseness.

 In order to ease our computations,   we assume that the function $H$ satisfies the following conditions:
\begin{hyp}\label{hyp.vol}
Let $H$ be the kernel appearing in the definition \eref{eqn.vol} of $Z$. We assume the following holds true:

(i) $H$ is a continuous function on $[0, T]^{2}$.

(ii) $H$ admits the following decompositions:
\begin{eqnarray}\label{eqn.H}
 H (t, s) =  \sum_{i=1}^{ \infty}  p_i(t) q_i (s),
\end{eqnarray}
where $p_{i}, q_{i} $ in \eref{eqn.H} are such that $p_{i}\in C^{1} ([0, T])$ and $q_{i} \in C([0, T])$, and where the    convergence in \eref{eqn.H} occurs in   $L_{1}([0,T]^{2})$.

(iii) For $n\geq 1$, set
\begin{eqnarray}\label{eqn.HL}
  H_n(t, s) =  \sum_{i=1}^{n}  p_i(t) q_i (s),  \quad
     L_{n}(t,s) = \frac{d}{dt} H_n(t, s)  =  \sum_{i=1}^{n}  p_i'(t) q_i (s).
\end{eqnarray}
Then   $L_{n} $   converges  in $L_{1}([0, T]^{2})$ to a continuous function $L(t, s)  $.
\end{hyp}

Following is the main result of this section:
\begin{thm}\label{thm.vol}
Consider the signal-observation equation \eref{eqn.vol}. Suppose that $H$ satisfies Hypothesis \ref{hyp.vol}.
For   $0\leq t\leq r$
we define the following augmented observation:
\begin{eqnarray*}
 V_{r,t} =  \lc \begin{array}{c}
X_{t} \\ \int_{0}^{t}L(r,s)X_{s}ds
  \end{array} \rc,
\end{eqnarray*}
as well as the augmented coefficients
\begin{eqnarray*}
&&
    B_{r}(s) = \lc\begin{array}{c} a(s) \\ L (r,s)  \end{array}\rc  [I_{m}\quad 0] ,  \qquad \Si=\lc\begin{array}{c} I_{m}\\ 0 \end{array}\rc
,
\end{eqnarray*}
and an initial condition
\begin{eqnarray*}
 \quad P_{0} =  \lc\begin{array}{cc} \Si&0 \\ 0&0  \end{array}\rc.
\end{eqnarray*}
  Then the
  conditional mean
   $\hat{V}_{r,t} = \EE [V_{r,t}|\cf^{Z}_{t}] $ 
   satisfies the equation
\begin{eqnarray}\label{eqn.vol.V}
\hat{V}_{r,t} = \hat{V}_{r,0}+ \int_{0}^{t} \lc\begin{array}{c} a(s)\\ L(r,s) \end{array}\rc \hat{X}_{ s} ds + \int_{0}^{t}\cp_{r,s}[H(s,s)\quad I]^{T}d \nu,
\end{eqnarray}
where $\nu_{t} = Z_{t} - \int_{0}^{t} [H(s,s)\quad I] \hat{V}_{r,s}ds $,  and where the conditional covariance $ \cp_{r,t}  =   \EE( (V_{r, t}-\hat{V}_{r, t}) (V_{t,t}-\hat{V}_{t,t})^{T} | \cf_{t}^{Z} )$ verifies the Riccati type equation
\begin{eqnarray}\label{eqn.vol.P}
\cp_{r,t}- P_{0}=  \int_{0}^{t} \lp \cp_{r,s}^{T} B_{t}(s)^{T} + B_{r}(s)\cp_{r,s}  +
\Si\Si^{T}
   -   \cp_{r,s}  \lc H(s,s) \quad I \rc^{T}\lc  H(s,s) \quad I \rc     \cp_{t,s}^{T}  \rp ds .
\end{eqnarray}
\end{thm}
\begin{proof}
We proceed according to the approximation given in Hypothesis \ref{hyp.vol}. This will be done in two steps.

\noindent \emph{Step 1: High-dimensional augmented signal.} \quad
Consider the signal $X$ given by \eref{eqn.vol.signal}, as well as the following approximation of the observation:
\begin{eqnarray*}
 d X_t &=&   a(t)X_t   dt  +     d W_t ,
\\
    Z^{n}_t &=&     \int^t_0 H_{n} (t, s)  X_s ds
     +        N_s
     =
      \sum_{i=1}^{n} p_{i}(t) X^{i}_{t} +N_{s} ,
\end{eqnarray*}
where we have set 
\begin{eqnarray*}
 X^{i}_{t} =    \int_{0}^{t} q_{i}(s) X_{s}ds, \quad i=1,\dots, n.
\end{eqnarray*}
Then an elementary product rule allows to write
\begin{eqnarray}\label{eqn.zn}
Z^{n}_{t} =   \int_{0}^{t} H_{n}(s,s) X_{s}ds + \sum_{i=1}^{n}\int_{0}^{t} p_{i}'(s) X^{i}_{s} ds +N_{s}.
\end{eqnarray}

We now consider an augmented signal and some augmented coefficients as follows:
 \begin{eqnarray}\label{eqn.Ubar}
\bar{U}^{n}=\lc \begin{array}{c}   X^{1}\\ \vdots \\ X^{n} \end{array} \rc , \quad \bar{b}^{n}(t) = \lc \begin{array}{c}
 q_{1}(t)  \\ \vdots  \\ q_{n}(t)
 \end{array} \rc  , \quad \si = \lc \begin{array}{c}
 I_{m} \\ 0
 \end{array} \rc, \quad \bar{h}^{n} (t) =
 \lc \begin{array}{ccc}
     p_{1}'(t) & \cdots &  p_{n}'(t)
  \end{array} \rc ,
\end{eqnarray}
and we set
\begin{eqnarray}\label{eqn.Unbar}
U^{n} = \lc \begin{array}{c} X\\ \bar{U}^{n} \end{array} \rc , \quad b^{n}(t) = \lc \begin{array}{c}
a(t) \\ \bar{b}^{n}(t)
 \end{array} \rc \lc\begin{array}{cc} I_{m}&0  \end{array}\rc , \quad   h^{n} (t) =
 \lc \begin{array}{cccc}
   H_{n}(t, t) & \bar{h}^{n}(t)
  \end{array} \rc .
\end{eqnarray}
Then we obtain the following linear system for the observation $U^{n}$:
\begin{eqnarray}\label{eqn.Un}
dU^{n}_{t} &=& b^{n}(t) U^{n}_{t} dt + \si dW_{t}
.
\end{eqnarray}
In addition, it is easily seen that the process $Z^{n}$ defined by \eref{eqn.zn} verifies
\begin{eqnarray}\label{eqn.Zn2}
Z^{n}_{t} &=& \int_{0}^{t} h^{n}(s) U^{n}_{s}ds +N_{t}.
\end{eqnarray}

As in Theorem \ref{thm.linear}, equation \eref{eqn.Un} and \eref{eqn.Zn2} can now be seen as a classical  Kalman-Bucy filtering system.
Hence we can invoke \cite[Chapter 9]{X} again, which yields the following equation for    $\hat{U}^{n}_{t} = \EE(U^{n}  | \cf^{Z^{n}}_{t}) $:
\begin{eqnarray}\label{eqn.Unhat}
\hat{U}^{n}_{t} &=& \hat{U}^{n}_{0} + \int_{0}^{t} b^{n}(s) \hat{U}^{n}_{s}ds + \int_{0}^{t}   P^{n}_{s} h^{n}(s)^{T} d\nu^{n},
\end{eqnarray}
where
$\nu^{n}_{t} = Z^{n}_{t} - \int_{0}^{t} h^{n} (s) \hat{U}^{n}_{s}ds $ is the corresponding innovation process. As far as the covariance function
  $P^{n}_{t}  = \EE( (U^{n}_{t}-\hat{U}^{n}_{t}) (U^{n}_{t}-\hat{U}^{n}_{t})^{T}   )$ is concerned,  it  satisfies the following Riccati equation:
\begin{eqnarray}\label{eqn.cov.approx}
P^{n}_{t} -P^{n}_{0} = \int_{0}^{t} \lp P^{n}_{s} b^{n}(s)^{T} + b^{n}(s)P^{n}_{s} + \si\si^{T} -   P^{n}_{s}h^{n}(s)^{T} h^{n}(s) P^{n}_{s}  \rp ds    .
\end{eqnarray}


 \noindent\emph{Step 2: Dimension reduction.}
 In Step 1, the dimension of the augmented signal $U^{n}$ grows with $n$. In order to go back to a low-dimensional signal, let us first compute the quantity $ h^{n}(t)U^{n}_{t}$ in \eref{eqn.Zn2}. Thanks to the definition \eref{eqn.Unbar} of $h^{n}$ and $U^{n}$ we have
 \begin{eqnarray}\label{eqn.hU}
h^{n}(t)U^{n}_{t} &=& H_{n}(t,t) X_{t} + \bar{h}^{n}(t) \bar{U}^{n}_{t}
=H_{n}(t, t)X_{t} +  \sum_{i=1}^{n} p_{i}'(t) \int_{0}^{t} q_{i}(s)X_{s}ds
\nonumber
\\
&=&H_{n}(t, t)X_{t} + \int_{0}^{t}L_{n}(t,s) X_{s}ds,
\end{eqnarray}
where the second equality is due to the definition \eref{eqn.Ubar} of $\bar{h}^{n}$ and the last equality stems from \eref{eqn.HL}.
Interestingly enough,
  equation \eref{eqn.hU}  suggests     to consider the filtering for the signal
 \begin{eqnarray*}
R_{t} :=  \sum_{i=1}^{\infty} p_{i}'(t) \int_{0}^{t} q_{i}(s)X_{s}ds  = \int_{0}^{t}L(t,s)X_{s}ds .
 \end{eqnarray*}
 To this aim, we consider a new process $R^{n}_{r,t}$ defined    for $0\leq t\leq r$ by
  \begin{eqnarray*}
R_{r,t}^{n} =  \sum_{i=1}^{n} p_{i}'(r) \int_{0}^{t} q_{i}(s)X_{s}ds  =    \bar{h}^{n}(r) \bar{U}^{n}_{t} ,
\end{eqnarray*}
  and we consider the following augmented signal (notice that the argument of $\bar{h}^{n}$ is frozen to $r$ in the equation below; see Remark):
 \begin{eqnarray}\label{eqn.Vn}
  V^{n}_{r,t} : = \lc \begin{array}{c}
X_{t} \\ R^{n}_{r,t}
  \end{array} \rc  = \lc \begin{array}{cc} I & 0 \\ 0&\bar{h}^{n}(r)  \end{array} \rc U^{n}_{t}.
\end{eqnarray}

We will now get the filtering equations for the augmented signal $V^{n}_{r,t}$. In order to get the equation for the conditional variance, we set
\begin{eqnarray*}
 \cp^{n}_{r,t}=   \EE( (V^{n}_{r, t}-\hat{V}^{n}_{r, t}) (V^{n}_{t,t}-\hat{V}^{n}_{t,t})^{T} | \cf_{t}^{Z^{n}} ) = \lc \begin{array}{cc} I & 0 \\ 0&\bar{h}^{n}(r)  \end{array} \rc P^{n}_{t} \lc \begin{array}{cc} I & 0 \\ 0&\bar{h}^{n}(t)  \end{array} \rc^{T}
 ,
\end{eqnarray*}
 where the second relation is obtained thanks to the definition \eref{eqn.Vn} of $V^{n}$ and the fact that $P^{n}_{t} = \EE [(U^{n}_{t}-\hat{U}^{n}_{t})(U^{n}_{t}-\hat{U}^{n}_{t})^{T}]$.

Hence multiplying relation \eref{eqn.cov.approx} by   $\lc \begin{array}{cc} I & 0 \\ 0&\bar{h}^{n}(r)  \end{array} \rc$ on the left and by $\lc \begin{array}{cc} I & 0 \\ 0&\bar{h}^{n}(t)  \end{array} \rc^{T}$ on the right,
we get
\begin{eqnarray*}
&&\cp^{n}_{r,t}-P_{0}
\\
&&\quad =  \int_{0}^{t} \lp (\cp^{n}_{r,s})^{T} B_{t}^{n}(s)^{T} + B_{r}^{n}(s)\cp^{n}_{r,s} +
\Si\Si^{T}
   -   \cp^{n}_{r,s}  \lc H_{n}(s,s) \quad I \rc^{T}\lc  H_{n}(s,s) \quad I \rc     (\cp^{n}_{t,s})^{T}  \rp ds ,
\end{eqnarray*}
where the coefficients $B^{n} $ and $\Si$ are defined by
\begin{eqnarray*}
B_{r}^{n}(s) = \lc\begin{array}{c} a(s) \\ L_{n}(r,s)  \end{array}\rc  [I_{m}\quad 0] ,
\quad
\Si=\lc\begin{array}{c} I_{m}\\ 0 \end{array}\rc
.
\end{eqnarray*}
Sending $n \to  \infty$ on both sides of the equation and applying Lemma \ref{lem.gau} we obtain   equation~\eref{eqn.vol.P} for $\cp_{r,t}$.
 The equation \eref{eqn.vol.V} for $\hat{V}_{r,t}$ can be derived in a  similar way by multiplying~\eref{eqn.Unhat} by the proper factor given by \eref{eqn.Vn}.  This completes the proof.
\end{proof}


We now give some details about our auxiliary result needed in the proof of Theorem \ref{thm.vol}.
Let us now consider the following auxiliary result:
\begin{lemma}\label{lem.gau}
Let $(x , y)$, $( \tilde{x}, y)$, $(x , y_n )$ and $(\tilde{x}, y_n)$ be joint Gaussian random vectors, and suppose that $  y_n $ converges to $  y  $ in $L^2(\Omega)$. Then

 (i) $\mathbb{E}  [ x | y_n  ] \rightarrow \mathbb{E}  [ x | y   ] $  in $L^2(\Omega)$;

 (ii) $ \mE   [   \mathbb{E}  ( x  | y_n   ) \mathbb{E}  ( \tilde{x} | y_n   )    ] \rightarrow   \mE   [   \mathbb{E}  ( x  | y    ) \mathbb{E}  ( \tilde{x} | y    )    ] $.
\end{lemma}

\begin{proof}
Denote $(x, y) \sim N(\mu , \Sigma)$, $(\tilde{x}, y) \sim N(\tilde{\mu} , \wt{\Sigma})$, $(x, y_n) \sim N(\mu^n , \Sigma^n)$,   $(\tilde{x}, y) \sim N(\tilde{\mu}^n , \wt{\Sigma}^n)$. Since $  y_n $ converges to $  y  $  in $L^2(\Omega)$, we have as $n \rightarrow + \infty$,

\begin{eqnarray*}
 \mE (y_n) \rightarrow  \mE (y) \quad \text{  and   ~ }\quad\mE \left\{    \lp
\begin{array}{c}
  x\\
     y_n
  \end{array} \rp
 \lp \begin{array}{cc} x& y_n \end{array}\rp
 \right  \}
   \rightarrow    \mE \left\{    \lp
\begin{array}{c}
  x\\
     y
  \end{array}
 \rp
 \lp \begin{array}{cc} x& y  \end{array}\rp
  \right  \} ,
 \end{eqnarray*}
which implies $ (\mu^n , \Sigma^n)  \rightarrow  (\mu  , \Sigma  ) $. On the other hand, writing 
$$\mE   \lc\lp
\begin{array}{c}
  x\\
     y
  \end{array}
\rp \rc =    \lp
\begin{array}{c}
  \mu_{1}\\
     \mu_{2}
  \end{array}
 \rp  ,  
 \quad\text{and}\quad 
 {\rm Cov}   \lc\lp
\begin{array}{c}
  x\\
     y
  \end{array}
  \rp
 \rc = \lp\begin{array}{cc} \Si_{11}&\Si_{12}\\ \Si_{21}& \Si_{22} \end{array}\rp
 ,
 $$ it is well known that
\begin{eqnarray*}
  \mE  (x | y)   = \mE[x] + \Sigma_{12} \Sigma_{22}^{-1} (y-  \mE[y]).
 \end{eqnarray*}
Taking into account the above two points, we have (i) and (ii).
\end{proof}


 \section{Application to Radar tracking}
 
 In this section we will apply the anticipative linear filter described in Section \ref{section.linear} to a standard practical problem considered in the literature. Specifically, we consider an anticipative version of the radar tracking system given in \cite[Chapter 5]{GA}. We shall observe how the algorithm induced by Theorem \ref{thm.linear} improves the estimation, versus a method using the classical Kalman filter and ignoring the anticipative problem. 
 
 In the radar tracking situation taken from \cite{GA} the signal $X$ is governed by equation \eref{eqn.linear1}, where $a(s)$ is a constant matrix. Namely, $X_{t}= [\begin{array}{cccccc} r_{t} & \dot{r}_{t} &  u^{1}_{t} & \theta_{t} & \dot{\theta}_{t} & u^{2}_{t} \end{array}]$ is a $6$-dimensional process, where $(r_{t}, \theta_{t})$ describes the position of the tracked vehicle expressed  in polar coodinates in $\RR^{2}$ ($r$ is called range and $\theta$ is called bearing in \cite{GA}). The coordinates $(u^{1}_{t}, u^{2}_{t})$ also stand for a maneuvering-correlated state noise, while $\dot{r}_{t}$ and $\dot{\theta}_{t}$ respectively represent the time derivatives for the range and the bearing.  
 

In \cite{GA} the dynamics for the process $X$ is supposed to be governed by the following equation: 
 \begin{align}\label{eqn.radar-s}
&dX_{t}  = \lc \begin{array}{cccccc}
0&1&0&0&0&0
\\
0&0&1&0&0&0
\\
0&0&\ka-1&0&0&0
\\
0&0&0&0&1&0
\\
0&0&0&0&0&1
\\
0&0&0&0&0&\ka-1
\end{array}\rc X_{t}dt +
\lc \begin{array}{cc}
 0&0
 \\
 0&0
 \\
 \si_{1} &0
 \\
 0&0
 \\
 0&0
 \\
 0&\si_{2}
\end{array}\rc
d W_{t},
\end{align}
where $\kappa$ is a real valued constant, $\si_{1}$, $\si_{2}>0$ and $W$ is a $2$-dimensional Brownian motion.  Equation \eref{eqn.radar-s} can be interpreted in the following way: we write that $r_{t}=r_{0}+\int_{0}^{t} \dot{r_{s}} ds$, 
where the velocity 
 $\dot{r}_{t} $ is equal to  $ u^{1}_{t} $ and $u^{1}_{t}$ is an Ornstein-Uhlenbeck process driven by $W^{1}_{t}$. Similar assumptions are also in order for the bearing $\theta$. 

As far as the observation process is concerned, we write equation \eref{eqn.linear} under the following form: 
\begin{align}\label{eqn.radar-o}
dZ_{t} =
\lc \begin{array}{cccccc}
\si_{\theta}^{-1}&0&0&0&0&0
 \\
0&0&0&\si_{\theta}^{-1}&0&0
\end{array}\rc X_{t}dt +
dN_{t},
\end{align}
where $\si_{\theta}$ is a positive constant and $N$ is a $2$-dimensional Brownian motion independent of~$W$.  Note that according to \eref{eqn.radar-o}, $Z^{1}_{t}$ (resp. $Z^{2}_{t}$) is a linear function of $r_{t} $ (resp. $\theta_{t}$) plus a noisy perturbation:
\begin{align*}
dZ^{1}_{t} = \si_{\theta}^{-1} r_{t} dt + dN^{1}_{t}, \quad \text{and}\quad 
dZ^{1}_{t} = \si_{\theta}^{-1} \theta_{t} dt + dN^{2}_{t}. 
\end{align*}

The anticipative nature of our system is enclosed in the following assumption. 
We assume that 
\begin{align}\label{eqn.x0}
X_{0}=\xi+\eta, \quad \text{where} \quad \eta =\ga \lc\begin{array}{cc}
 1&0\\
 0&0\\
 1&0\\
 0&1\\
 0&0\\
 0&1
\end{array}\rc N_{1} .
\end{align}
In equation \eref{eqn.x0}  the vector $\xi$ is a standard $\RR^{6}$-valued Gaussian random variable independent of   $(W, N)$, and $\gamma $ is a positive constant measuring the anticipation strength.  
Notice that according to equation \eref{eqn.x0} the anticipation of $X_{0}$ is only felt on the components  
$r_{0}$, $\theta_{0}$ and $u_{0}$ of $X_{0}$. For our numerical simulations we take  
  $\ka=0.5$, 
 $\si_{\theta}  =   0.017~{\rm rad}  $, $\si_{1}  =  103/3 $   and $\si_{2} = 1.3 $.

As the reader might expect, 
 our new filter \eref{eqn.uh}-\eref{eqn.P} provides a much better estimation   for the anticipative signal-observation system
 \eref{eqn.radar-s}-\eref{eqn.radar-o}. 
 This is attested by the   following simulation.   Namely,  in the figures below the blue curve  
  represents the signal path, the yellow curve is drawn according to the classical Kalman filter, and the orange curve is drawn thanks to our new filter. We have zoomed in the picture for comparison purposes, so that the signal curve appears to be linear.

 \includegraphics[width=8.2cm, height=7.2cm]{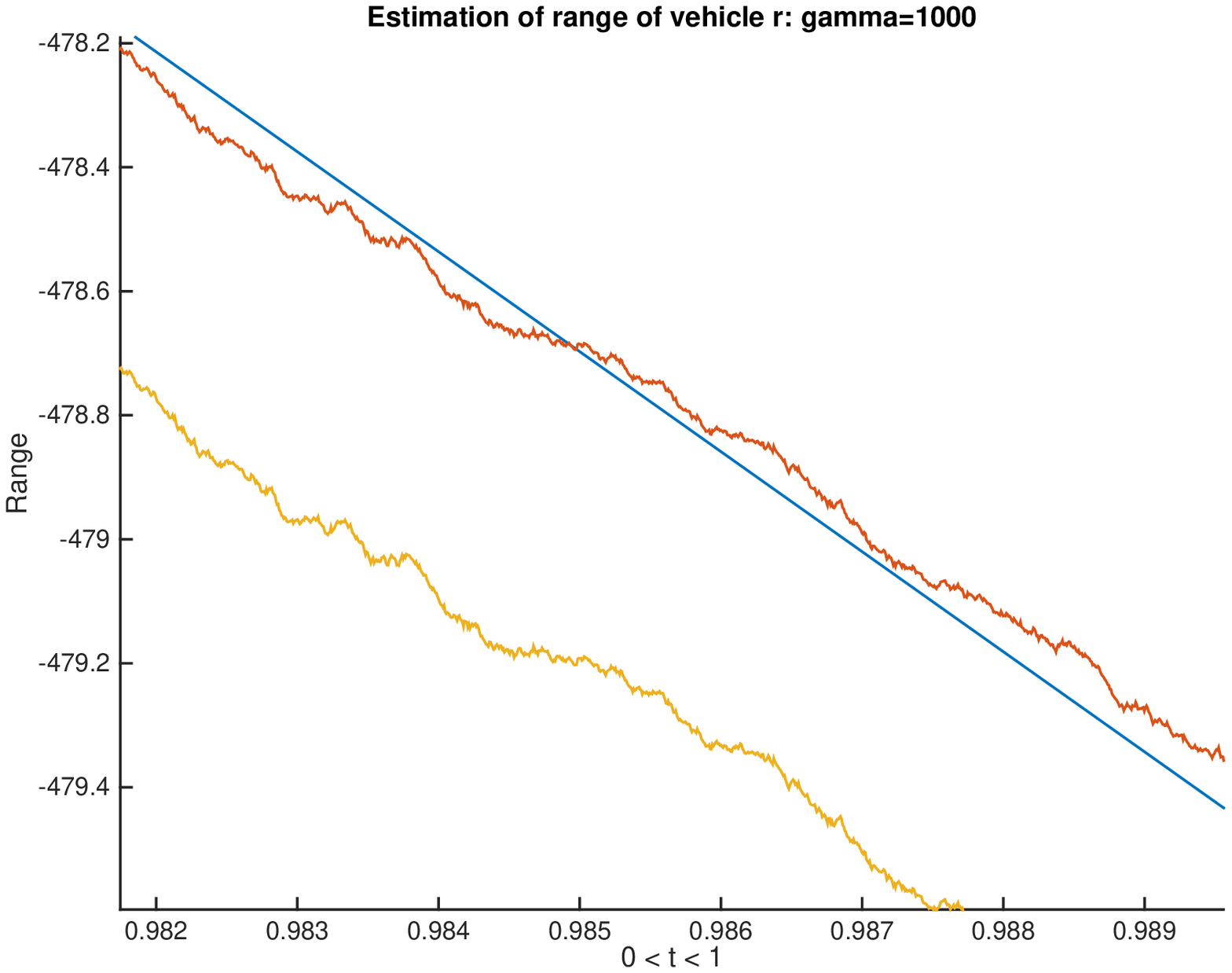}
 \includegraphics[width=8.2cm, height=7.2cm]{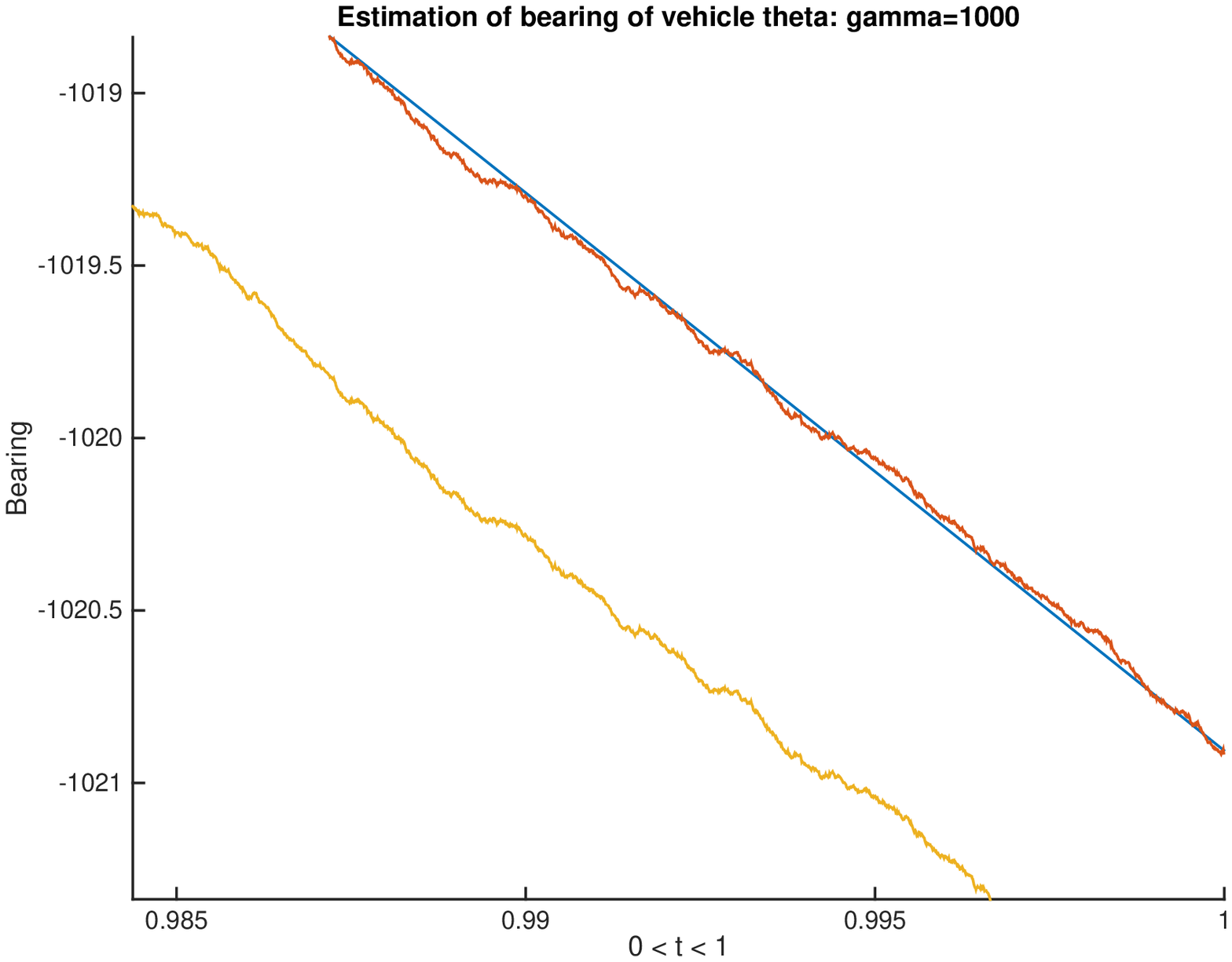}

In this set of figures we successively take   $\ga=1000$, $  100$, $10$, $1$, in order to observe the effect of the anticipation strength on the filter performance.  

 \includegraphics[width=8.2cm, height=6.4cm]{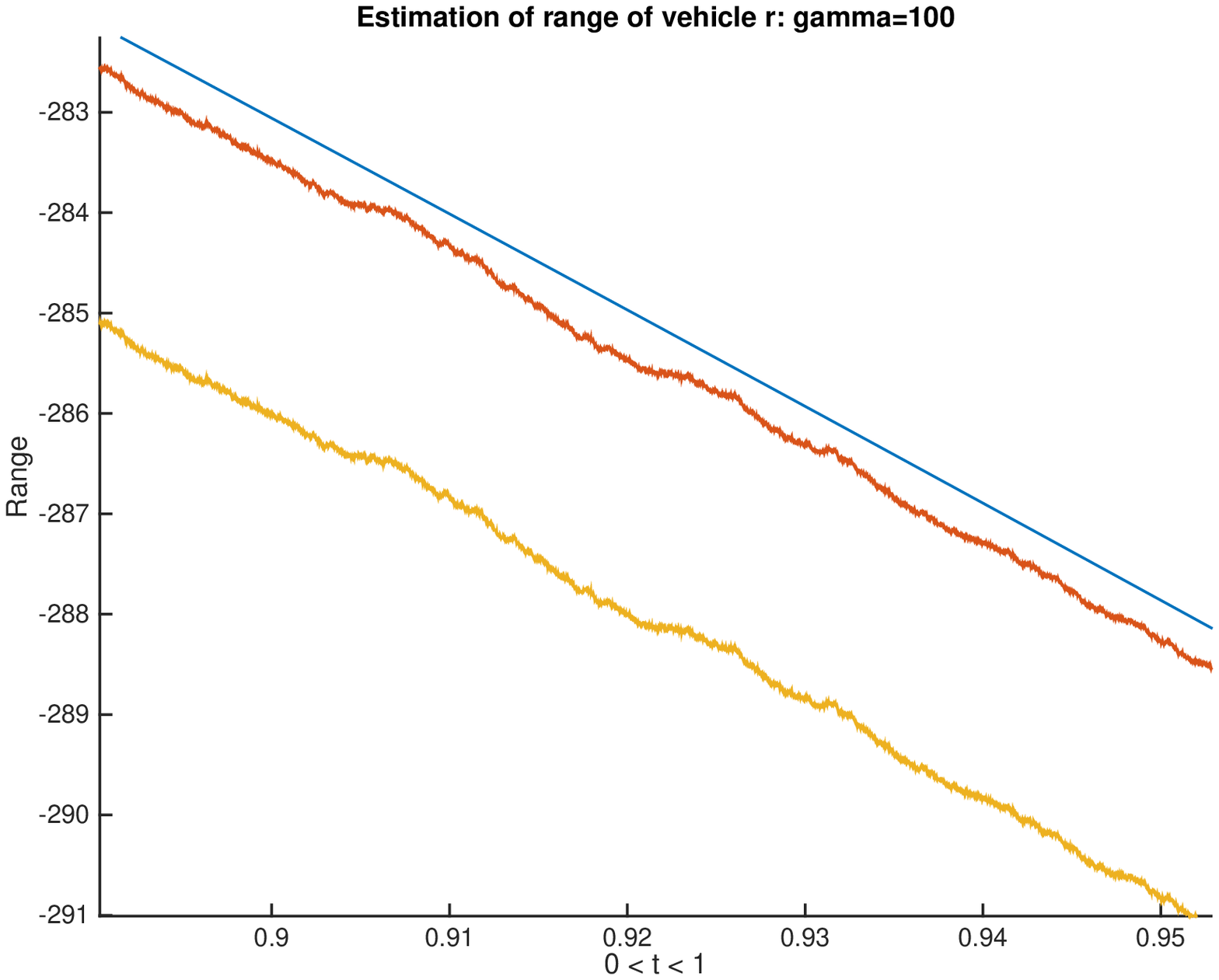}
 \includegraphics[width=8.2cm, height=6.4cm]{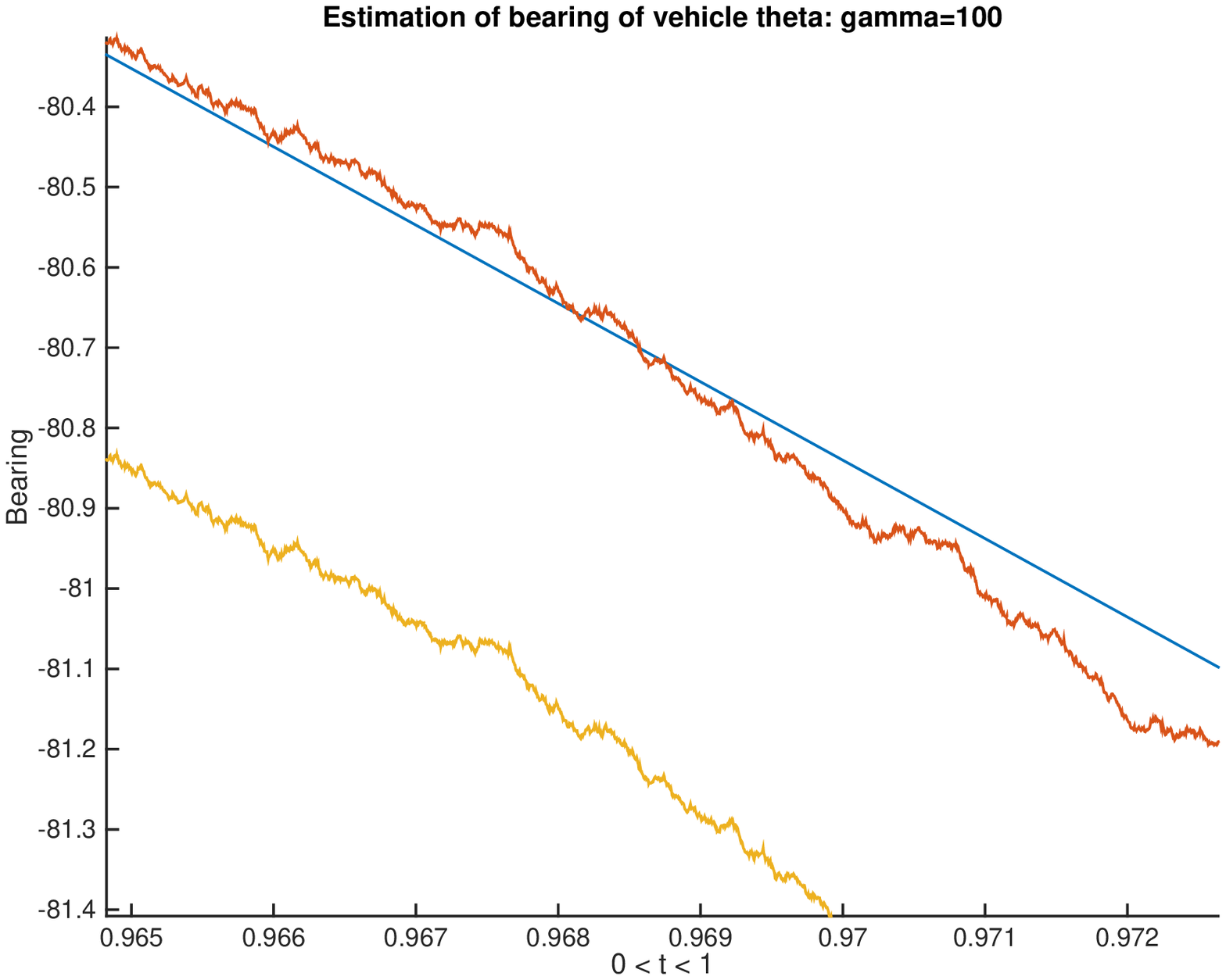}

 \includegraphics[width=8.2cm, height=6.4cm]{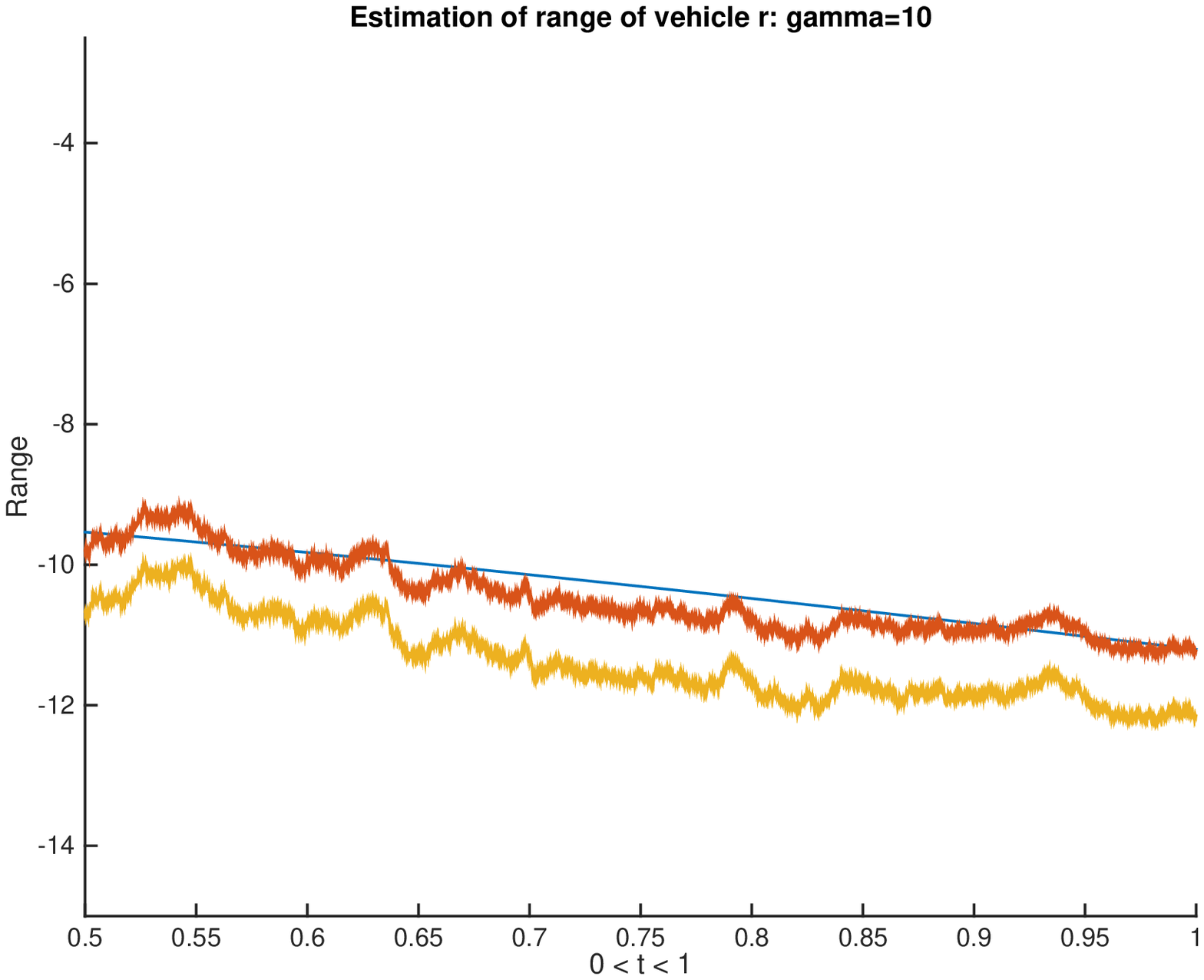}
 \includegraphics[width=8.2cm, height=6.4cm]{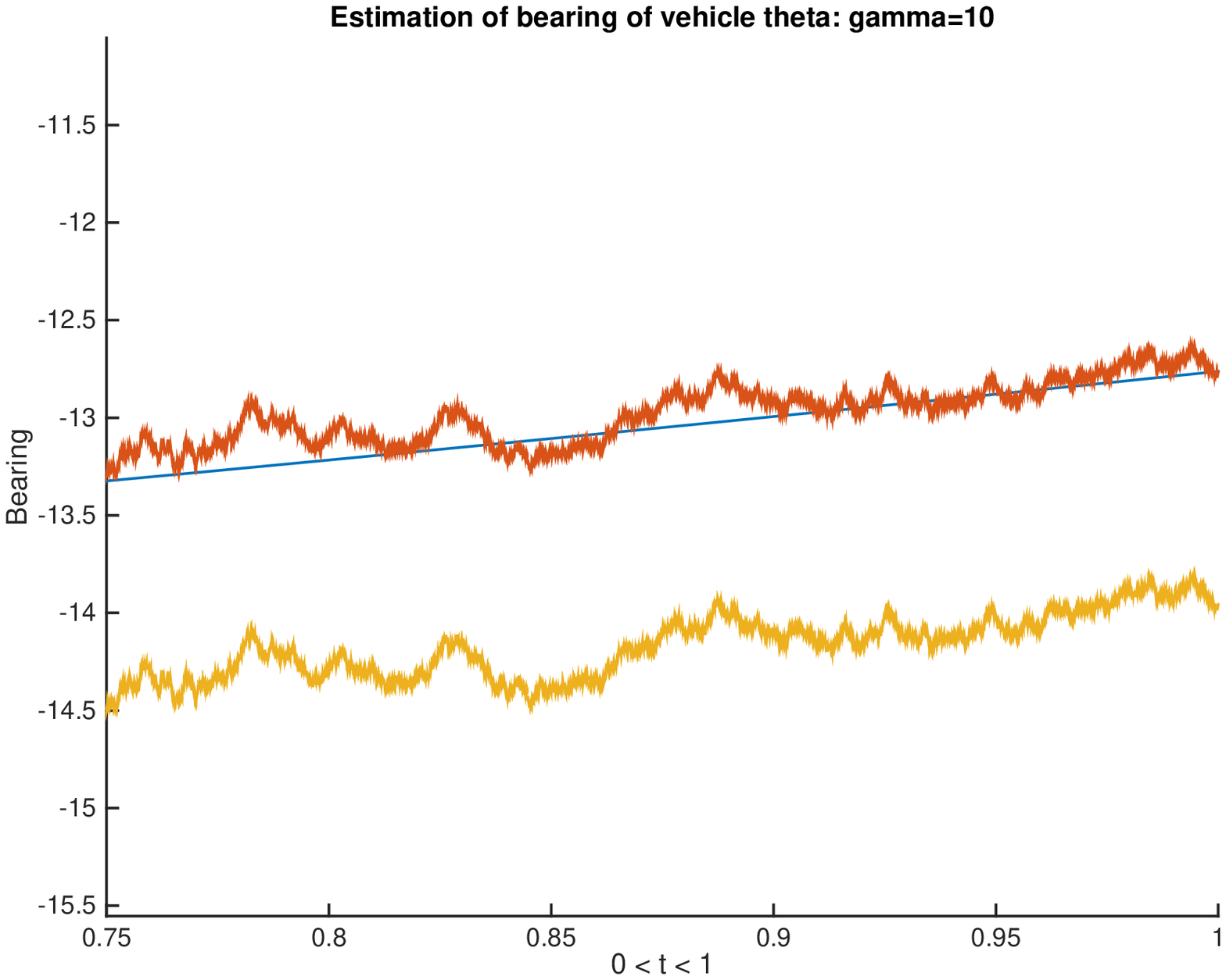}

 \includegraphics[width=8.2cm, height=7.2cm]{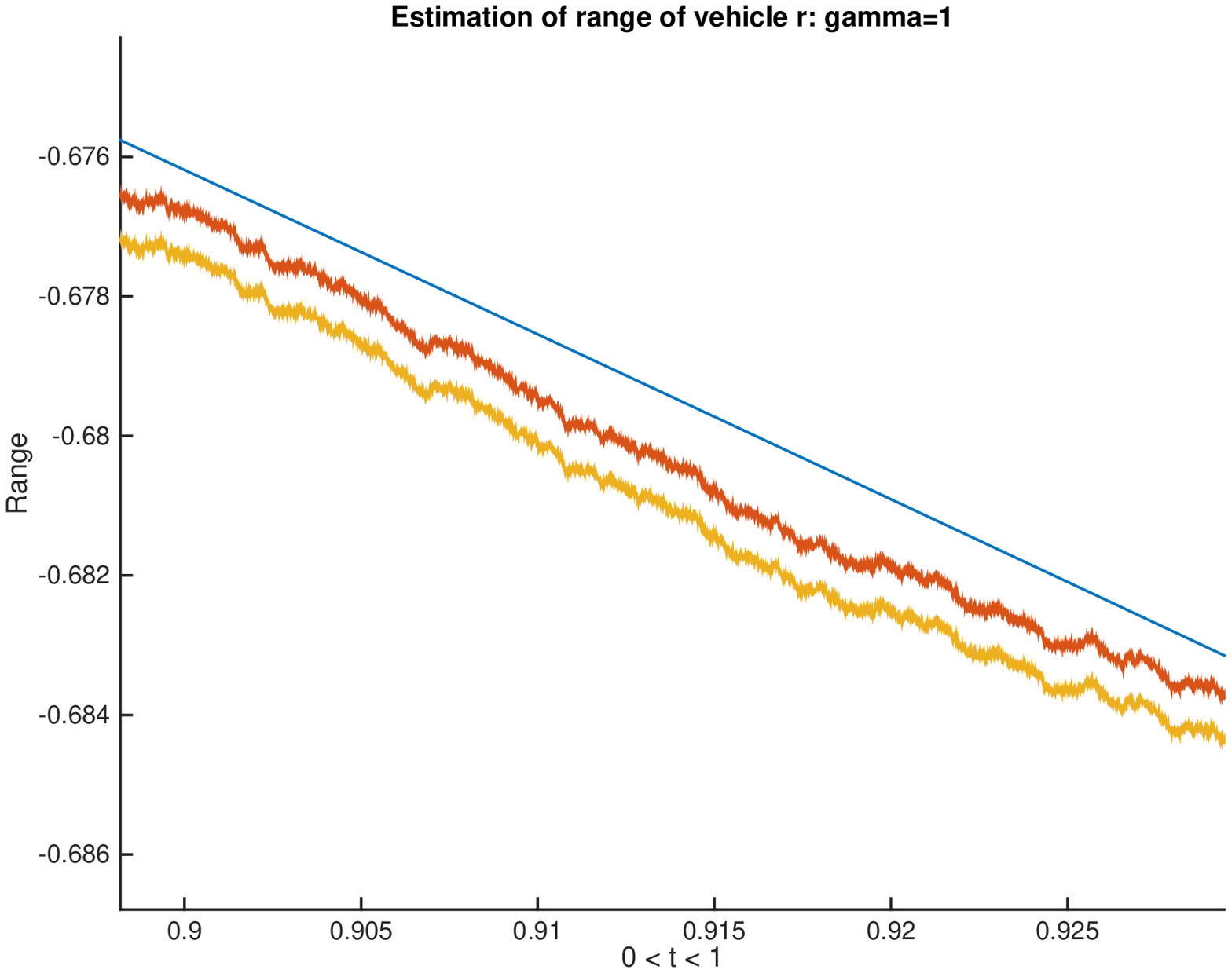}
 \includegraphics[width=8.2cm, height=7.2cm]{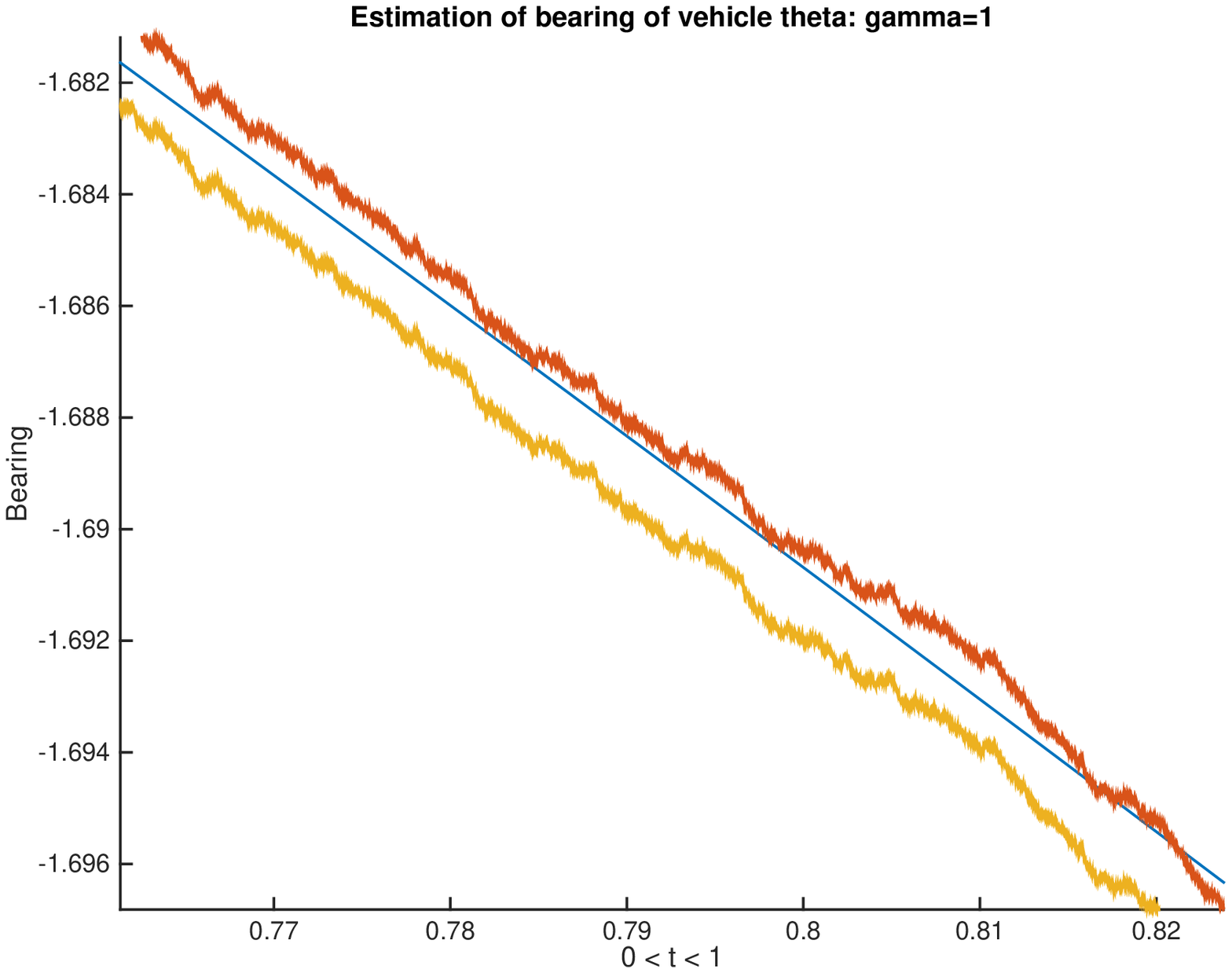}

 \begin{remark}
 As our simulation shows, the improvement of the new filter from the classical Kalman filter becomes more significant as the anticipation gets stronger.
 \end{remark}
  
Denote by $\hat{X}^{i}_{t}$, $i=1,\dots, 6$ the optimal filter, i.e. $\hat{X}^{i}_{t}=\mE(X^{i}_{t}|\cf^{Z}_{t})$, and denote by $\bar{X}^{i}_{t}$ the estimate obtained while  the anticipation is ignored. The following tables present the ratio $R_{i}$ between the error deviations of these two estimates at time $t $, that is,  
\begin{equation*}
R_{i}(t)=\lp\frac{\mE\lc |X_{t}^{i}-\hat{X}^{i}_{t}|^{2}\rc }{  \mE\lc|X_{t}^{i}-\bar{X}^{i}_{t}|^{2}\rc}\rp^{1/2} 
\end{equation*}
For $t=1$, we consider the   anticipation strength   $\ga = 1,10,100$.
  \begingroup
\setlength{\tabcolsep}{10pt} 
\renewcommand{\arraystretch}{1.5} 
  \begin{center}
   \begin{tabular}{|c|c|c|c|c|c|c|}
  \hline
t=1 &$R_{1}$&$R_2$&$ R_{3}$&$ R_{4}$&$ R_{5}$&$ R_{6}$\\ \hline  
 $\gamma=1$&$
 0.0100  $ &$  0.0361  $ &$  0.0608 $ &$   0.0141 $ &$   0.0400  $ &$  0.0616$
  \\ \hline
 $\gamma=10$&$
    0.0100 $ &$   0.0265 $ &$   0.0574   $ &$ 0.0100   $ &$ 0.0316 $ &$   0.0574$
 \\ \hline
 $\gamma=100$&$   0.0100  $ &$  0.0265 $ &$   0.0574  $ &$  0.0100  $ &$  0.0316  $ &$  0.0574$
 \\ \hline
\end{tabular}
\end{center}
\endgroup
\noindent
For $t=3/4$, we consider the anticipation strength $\gamma=1, 10, 100, 1000$. 
    \begingroup
\setlength{\tabcolsep}{10pt} 
\renewcommand{\arraystretch}{1.5} 
  \begin{center}
   \begin{tabular}{|c|c|c|c|c|c|c|}
  \hline
t=3/4 &$R_{1}$&$R_2$&$ R_{3}$&$ R_{4}$&$ R_{5}$&$ R_{6}$
  \\ \hline
 $\gamma=1$&$  0.3670  $ &$  0.3684  $ &$  0.3688  $ &$  0.3670  $ &$  0.3686 $ &$   0.3689$
 \\ \hline
 $\gamma=10$&$   0.4603 $ &$   0.4609  $ &$  0.4615 $ &$   0.4574  $ &$  0.4610  $ &$  0.4615$
 \\ \hline
 $\gamma=100$&$   0.4965  $ &$  0.4969 $ &$   0.4981  $ &$  0.4953 $ &$   0.4970 $ &$   0.4981$
 \\ \hline
 $\gamma=1000$&$   0.5007 $ &$   0.5011   $ &$ 0.5023  $ &$  0.4997 $ &$   0.5012  $ &$  0.5023$
  \\ \hline
\end{tabular}
\end{center}
As the reader can see, our ratios are small regardless of the values of $t$ and $\ga$. This indicates that our  filter performs well compared with a filter ignoring the anticipative nature of the signal.
\endgroup

\

 We end this section with a remark  on the   stability of our new filter  \eref{eqn.uh}-\eref{eqn.P}:

 \begin{remark}
 In our numerical experiments, we find that the stability of the original system \eref{eqn.linear1}-\eref{eqn.linear} and the new system \eref{eqn.nlinear} can be quite different.
 As $\eta$ gets smaller the stability of the new system usually decreases,  and therefore finer mesh is needed in the simulations in order to capture the accuracy improvement achieved by our new filter. 
  \end{remark}

\section*{Acknowledgments}
Y. Liu wish to thank Professor Yaozhong Hu and Professor David Nualart for helpful discussions. 
G. Lin gratefully acknowledges the support from National Science Foundation (DMS-1555072, DMS-1736364, and DMS-1821233).
S. Tindel gratefully acknowledges the support from National Science Foundation DMS-1613163.


\end{document}